\documentclass[a4paper,12pt]{article}
\usepackage[T1]{fontenc}
\usepackage{amsmath}
\usepackage{amssymb}
\usepackage{graphicx}
\usepackage{color}
\usepackage{bbold}
\usepackage{dsfont}

\renewcommand{\P}{{\rm P}}
\newcommand{\E}{{\rm E}}
\newcommand{\RR}{\mathbb{R}}
\newcommand{\CC}{\mathbb{C}}

\newcommand{\tr}{^{\textnormal {\tiny  T}}}

\newcommand{\diag}{\mathrm{diag}}

\newcommand{\re}{\mathrm{Re}}
\newcommand{\ud}{\mathrm{d}}  
\newcommand{\indic}{\mathbb 1}

\renewcommand{\u}{{\mbox{\tiny $+$}}}
\renewcommand{\d}{{\mbox{\tiny $-$}}}

\newcommand{\FF}{\mathcal{F}}
\newcommand{\Ss}{\mathcal{S}}
\newcommand{\Su}{\mathcal{S}_\u}
\newcommand{\Sd}{\mathcal{S}_\d}
\newcommand{\Sz}{\mathcal{S}_0}

\newcommand{\PP}{\mathcal{P}}
\newcommand{\UU}{\mathcal{U}}

\newcommand{\vect}[1]{\boldsymbol #1}

\newcommand{\valpha}{\vect \alpha}
\newcommand{\vrho}{\vect \rho}
\newcommand{\vbeta}{\vect \beta}
\newcommand{\veta}{\vect \eta}

\newcommand{\vone}{\vect 1}
\newcommand{\vzero}{\vect 0}

\newcommand{\vcc}{\vect c}

\newcommand{\vG}{\vect G}
\newcommand{\vh}{\vect h}
\newcommand{\vm}{\vect m}
\newcommand{\vp}{\vect p}

\newcommand{\vu}{\vect u}
\newcommand{\vv}{\vect v}
\newcommand{\vw}{\vect w}

\newcommand{\bphi}{\bar\varphi}

\newcommand{\wM}{\widetilde M}

\newcommand{\whH}{\widehat{H}}
\newcommand{\whU}{\widehat{U}}
\newcommand{\whK}{\widehat{K}}
\newcommand{\whpsi}{\widehat{\Psi}}
\newcommand{\whphi}{\widehat{\Phi}}

\newcommand{\vligne}[1]{\begin{bmatrix} #1 \end{bmatrix}}

\newtheorem{defn}{Definition}[section]
\newtheorem{lem}[defn]{Lemma}

\newtheorem{thm}[defn]{Theorem}
\newtheorem{cor}[defn]{Corollary}
\newtheorem{rem}[defn]{Remark}

\newcommand{\qed}{\hfill $\square$}
\newenvironment{proof}{
      \noindent {\bf Proof }}{\qed
      \vspace{0.25\baselineskip}
}
\newcommand{\debproof}{\begin{proof}}
\newcommand{\finproof}{\end{proof}}

\definecolor{darkmagenta}{rgb}{0.5,0,0.5}
\definecolor{darkgreen}{rgb}{0,0.6,0}
\definecolor{darkblue}{rgb}{0,0,0.6}
\definecolor{darkred}{rgb}{0.8,0,0}
\definecolor{mellow}{rgb}{.847, 0.72, 0.525}

\usepackage{pifont}

\usepackage[normalem]{ulem}

\begin{document}

\title{Analysis of fluid flow models}

\author{Guy Latouche\thanks{Universit\'e libre de Bruxelles,
D\'epartement d'informatique, CP 212, Boulevard du Triomphe, 1050
Bruxelles, Belgium, \texttt{latouche@ulb.ac.be}.}
\and
Giang T. Nguyen\footnote{The University of Adelaide, School of Mathematical Sciences, SA 5005, Australia, \texttt{giang.nguyen@adelaide.edu.au}}}
\date{\today}
\maketitle

\begin{abstract}

  Markov-modulated fluids have a long history.  They form a simple
  class of Markov additive processes, and were initially developed in
  the 1950s as models for dams and reservoirs, before gaining much
  popularity in the 1980s as models for buffers in telecommunication
  systems, when they became known as {\em fluid queues}.  More recent
  applications are in risk theory and in environmental studies.

  In telecommunication systems modelling, the attention focuses on
  determining the stationary distribution of the buffer content.
  Early ODE resolution techniques have progressively given way to
  approaches grounded in the analysis of the physical evolution of the
  system, and one only needs now to solve a Riccati equation in order
  to obtain several quantities of interest.  To the early algorithms
  proposed in the Applied Probability literature, numerical analysts
  have added new algorithms, improved in terms of convergence speed,
  numerical accuracy, and domain of applicability.

  We give here a high-level presentation of the matrix-analytic
  approach to the analysis of fluid queues, briefly address
  computational issues, and conclude by indicating how this has been
  extended to more general processes.

\end{abstract}

\section{Introduction}
   \label{s:intro}

Markov-modulated processes are popular because they can be used to describe
the evolution of simple systems under conditions that vary in time.  The
fluid flow processes presented here have long found applications as
models for dams and reservoirs (Loynes~\cite{loyne62}), and for buffers in
telecommunication systems, Anick {\it et
  al.}~\cite{ams82} being a famous early paper.   It is in the latter
context that the term {\em fluid queue} was
coined.   Later, the domain of
applicability of fluid flows has been extended to risk theory (Avram and
Us\'abel~\cite{au03b}, Badescu {\it et al.}~\cite{bdlrs03}),
operations management (Bean {\it et al.}~\cite{bos10}) and others.

Fluid flows are two-dimensional processes
$\{X(t), \varphi(t): t \in \RR^+\}$, where $\{\varphi(t)\}$ is a
continuous-time, irreducible Markov chain on some finite state space
$\Ss = \{1, \ldots, m\}$, and $X(t)$ takes values in $\RR$ under the
control of $\varphi$.  In the simplest form,
\begin{equation}
   \label{e:x}
X(t) = X(0) + \int_0^t c_{\varphi(s)} \, \ud s,
\end{equation}
with $\vcc = [c_i : i \in \Ss]$ being a vector of arbitrary real constants.  The
component $X$ is called the level, $\varphi$ is called the phase, and
the level is a piecewise-linear function, with constant slope over
intervals when the phase is constant.  An example with four phases is
shown on Figure~\ref{f:mmff}.  

It will be useful in the sequel to
partition $\Ss$ into three subsets according to the sign of the rates $c_i$:
\begin{equation}
   \label{e:s}
\Su = \{i \in \Ss : c_i >0\},
\quad
\Sd = \{i \in \Ss : c_i <0\},
\quad
\Sz = \{i \in \Ss : c_i =0\}.
\end{equation}

\begin{figure}[t]
\centering{
\includegraphics{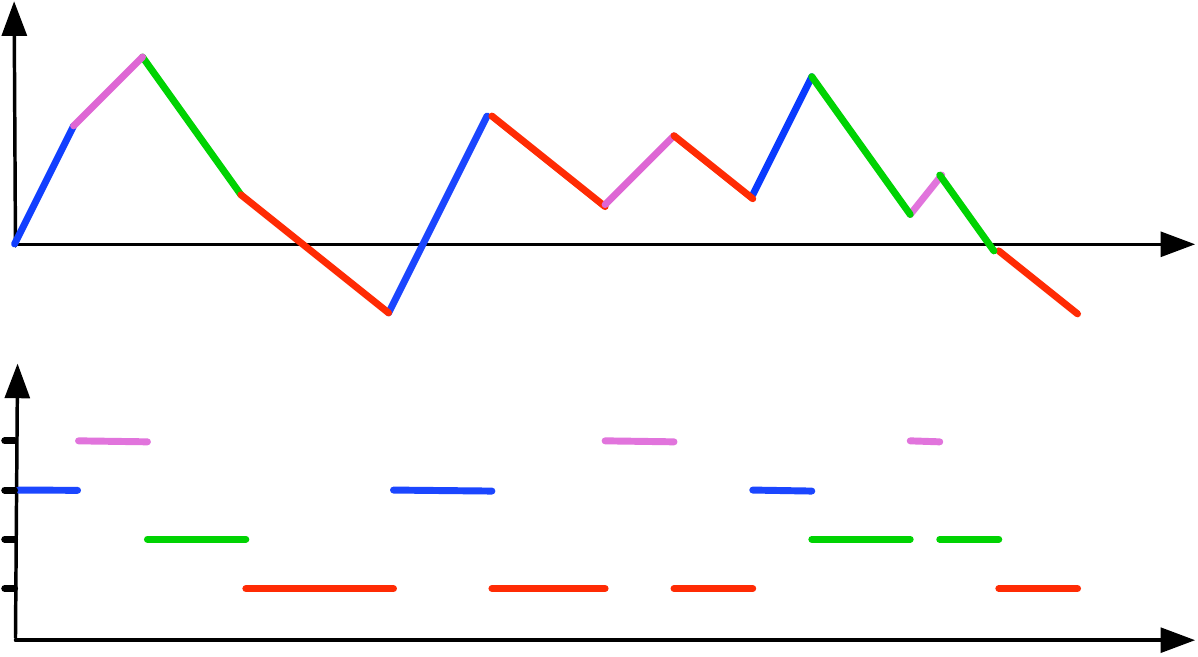} 
\put(-10,-10){\makebox(0,0)[c]{$t$}}
\put(-10,104){\makebox(0,0)[c]{$t$}}
\put(-357,185){\makebox(0,0)[c]{$X(t)$}}
\put(-355,80){\makebox(0,0)[c]{$\varphi(t)$}}
\put(-352,65){\makebox(0,0)[c]{$4$}}
\put(-352,50){\makebox(0,0)[c]{$3$}}
\put(-352,35){\makebox(0,0)[c]{$2$}}
\put(-352,20){\makebox(0,0)[c]{$1$}}
}
\caption{Sample trajectory of a fluid flow (top) and of its controlling Markov
chain (bottom).  The fluid rates vector is $\vcc = [-0.8, -1.4, 2, 1]$
so that different states correspond to different slopes.}
\hrulefill
\label{f:mmff}
\end{figure}
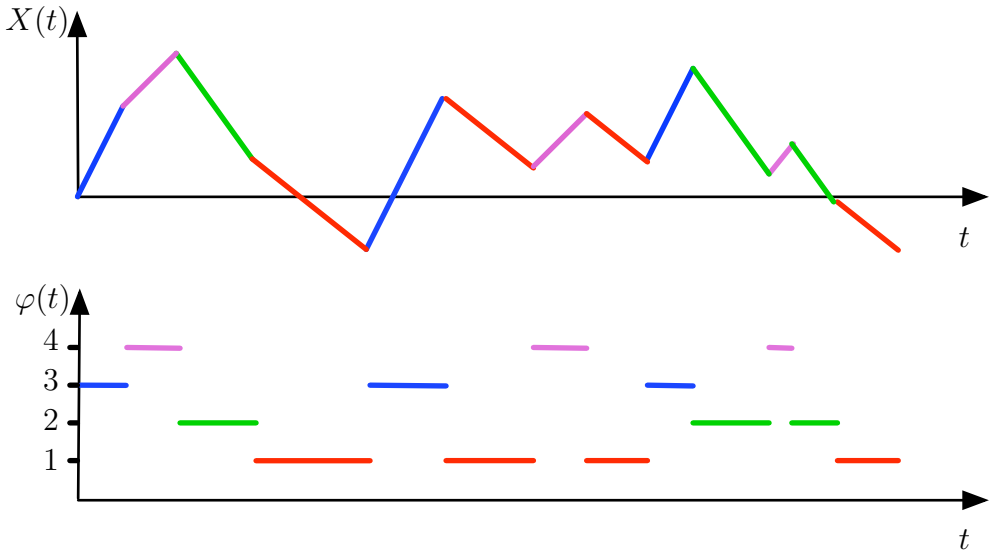

In many applications, $X$ is a model for a physical quantity (water in
a reservoir, packets in a buffer, energy level of a battery, etc.) and
may not take negative values.  In such cases, one might assume that
whenever the level becomes equal to 0 and the rate is negative, 
the level remains equal to 0 until there is a change to a phase with
positive rate.  This is illustrated in Figure~\ref{f:mmffb}: at time
$\theta_1$ the fluid hits level 0, the phase is equal to 1 with
$c_1 = -0.8$.  The fluid remains equal to 0 until time $\delta_1$ where
the phase process switches to state 3, with $c_3 =2$.

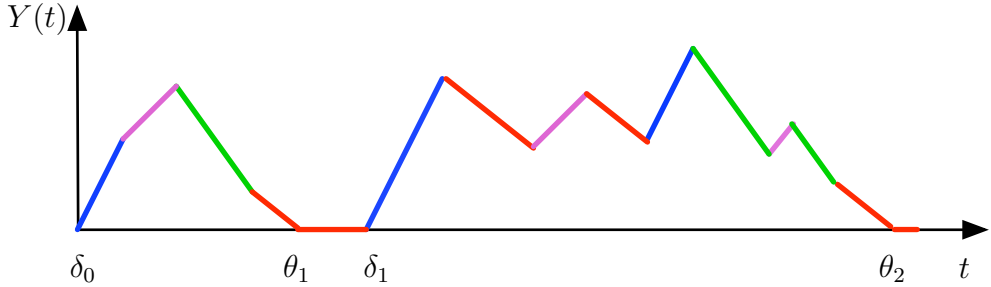
\begin{figure}[t]
\centering{
\includegraphics{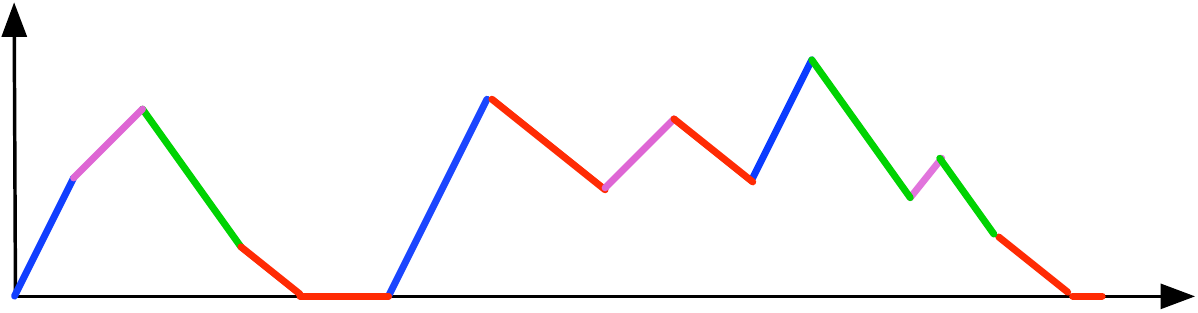} 
\put(-10,-10){\makebox(0,0)[c]{$t$}}
\put(-340,-10){\makebox(0,0)[c]{$\delta_0$}}
\put(-260,-10){\makebox(0,0)[c]{$\theta_1$}}
\put(-230,-10){\makebox(0,0)[c]{$\delta_1$}}
\put(-37,-10){\makebox(0,0)[c]{$\theta_2$}}
\put(-357,84){\makebox(0,0)[c]{$Y(t)$}}
}
\caption{\label{f:mmffb}
Sample trajectory of a regulated fluid flow}
\hrulefill
\end{figure}

Such a mechanism  is justified by the fact that   the rates
$c_i$ often result from the superposition of different effects, some which
remove fluid from the buffer and some which add fluid.  If at some
time the buffer is empty and the output rate remains greater than the
input rate, fluid does not accumulate and the buffer remains empty.  

Formally, we define the regulator  
$
R(t) = -\min(0, \min_{0 \leq s \leq t}X(s))
$
 and the fluid queue (the regulated fluid flow) is $\{Y(t), \varphi(t)\}$, with 
\begin{equation}
   \label{e:regulated}
Y(t) = X(t)+R(t).
\end{equation}

In the first part of this paper, we focus on characterising the
stationary distribution of the regulated fluid queue, when it exists.
Although the details of our presentation are very
much inspired by Ramaswami~\cite{ram99} and da Silva Soares and
Latouche~\cite{dssl01}, we follow a slightly different path and give
explicit reference to semi-regenerative processes
(\c{C}inlar~\cite[Chapter 10]{cinla75}); this allows us  to interpret in a
unified manner the ad-hoc analysis of several published variants of
our basic model.

Let us assume that $Y(0)=0$ and define the
sequences $\{\delta_n: n \geq 0\}$ and  $\{\theta_n: n \geq 1\}$ as
follows:
\begin{align*}
\delta_0 & = \inf\{t > 0: Y(t) >0\}, \\
\theta_n & = \inf\{t > \delta_{n-1}: Y(t) =0\}, \qquad
\delta_n  = \inf\{t > \theta_n: Y(t) >0\}
\end{align*}
(see an illustration in Figure~\ref{f:mmffb}).  It is easily seen that
$\{\theta_n: n \geq 1\}$ is a set of regenerative epochs for the
process $\{Y(t), \varphi(t)\}$: 

\begin{itemize}
\item 
$\{Y(\theta_n), \varphi(\theta_n)\}$ is a Markov chain on the
  state space $\{0\} \times \Sd$ as the fluid is reaching down to level
  zero at these epochs, in a phase of $\Sd$,
\item
the process over the interval $[\theta_n, \infty)$ is independent of
the process over the interval $[0, \theta_n)$, given
$\varphi(\theta_n)$,  for all $n$,  and
\item
the distribution of the process over the interval $[\theta_n,
\infty)$, given $\varphi(\theta_n)= s$,  is the same as the
distribution of the process over the interval $[\theta_1,
\infty)$, given $\varphi(\theta_1)= s$ for all $n$ and $s \in \Sd$.
\end{itemize}
In consequence, we may immediately write that the stationary
distribution $\vG(x) = \vligne{G_i(x): i \in \Ss}$, defined as%
\footnote{We use boldface letters to represent vectors, and capital
  letters for matrices.} 
\[
G_j(x) = \lim_{t \rightarrow \infty } \P [Y(t) \leq x, \varphi(t)= j],
\qquad j \in \Ss,
\]
is given by
\begin{equation}
   \label{e:G}
\vG\tr(x) = (\vrho\tr \vm)^{-1} \vrho\tr M(x)
\end{equation}
where 
\begin{itemize}
\item $\vrho$ is the stationary probability vector of
  $\{(Y(\theta_n),\varphi(\theta_n))\}$,
\item
$M(x)$ is an $\Sd$ by $\Ss$ matrix, with components $M_{ij}(x)$ equal
to the conditional expected sojourn time of $(Y(t), \varphi(t))$ in $[0,x] \times
\{j\}$ during a regeneration interval $[\theta_n, \theta_{n+1})$,
given that $\varphi(\theta_n)=i$, and
\item
 $\vm = M(\infty) \vone$, where $\vone$ represents a vector of
 ones; the components of $\vm$ are the conditional expected lengths of
 intervals between regeneration points, given the phase at the
 beginning of the interval.
\end{itemize}
We give in Section~\ref{s:regenerative} a few basic characteristics of
the process at level 0.  The vector $\vrho$ is determined in
Section~\ref{s:passage}, where we define and analyse two important
first passage probability matrices, and $M(x)$ is determined in
Section~\ref{s:crossings} through the number of crossings of a given
level during regenerative intervals.  The results in these two
sections are brought together in Section~\ref{s:stationary} to give
the stationary distribution of the fluid queue.

The key matrices defined in Sections~\ref{s:passage} 
and~\ref{s:crossings} have very distinct physical significance but they are
algebraically closely related.  We show this in
Section~\ref{s:factorisation}, using results originally proved in
Rogers~\cite{roger94}.  

In Section~\ref{s:times}, we characterise the {\em distribution} of
first passage times to a given level, and we analyse in
Section~\ref{s:escape} the first exit from an {\em interval}.
The two sections 
come as complements to
Section~\ref{s:passage}.

In many applications, in particular in telecommunication 
modeling, the buffer level is not allowed to grow without bounds.
Often, the evolution of the process changes as the upper or lower
boundary is reached.  These, and other modifications of the basic fluid flow
model, are briefly discussed in Section~\ref{s:extensions} where we
show how the regenerative approach may be readily adapted to more
complex assumptions.

One of the nice features of the matrix-analytic approach
is that computational
algorithms are easily constructed, following the development of the
theoretical results.  As an illustration, we give in Section~\ref{s:algorithms} two of the
simplest, and yet very efficient, algorithms for the numerical
computation of the key matrix $\Psi$ identified in
Section~\ref{s:passage}.

Most of the results presented here have appeared earlier.  For that
reason, we give explanatory justifications mostly, and we refer to
published sources for  formal justifications.  In a few cases,
however, we give formal proofs:  
in Section~\ref{s:escape} we give a new treatment of
escape probabilities for
null recurrent processes (Lemma~\ref{t:null} and 
Theorem~\ref{t:fullsolution}), and we offer with
Theorem~\ref{t:functional} a novel justification for a nice
computational procedure for $\Psi$.

\section{Preliminaries}
\label{s:regenerative}

We partition the generator $Q$ of the Markov process
$\{\varphi(t): t \in \RR^\u\}$ in a manner conformant to the
partition (\ref{e:s}) of $\Ss$ and write, possibly after a
permutation of rows and columns,%
\footnote{We write $\vv_\u$ for the subvector $[v_i: i \in \Su]$ of
  any vector $\vv$, and $M_{\u\u}$ for the submatrix of any matrix $M$
  at the intersection of the rows and columns in $\Su$.   Other
  sub-vectors and sub-blocks are similarly defined.}
\begin{equation}
   \label{e:q}
Q = \vligne{Q_{\u\u} & Q_{\u\d} & Q_{\u0} \\ Q_{\d\u} & Q_{\d\d}&
  Q_{\d0} \\Q_{0\u}&  Q_{0\d} & Q_{00} }.
\end{equation}
We assume that $Q$ is irreducible.  We also define the diagonal matrix
$C$ of fluid rates, 
$C = \diag(\vcc)$, and we partition it
as
\[
C = \vligne{C_\u \\ & C_\d \\ & & 0}.
\]
The fluid process $\{X(t)\}$ moves up and down in a random manner but
its general direction is determined by the stationary drift $\mu =
\valpha\tr \vcc$, where $\valpha$ is the stationary probability
vector of $Q$: $\valpha\tr Q= \vzero$, $\valpha\tr \vone =1$.  If $\mu
>0$, the process eventually drift to $+ \infty$, that is, $\lim_{t
  \rightarrow \infty} X(t) = \infty$,  if $\mu < 0$, then $\lim_{t
  \rightarrow \infty} X(t) = -\infty$; in both cases the process is
transient.  If $\mu=0$, the process is null-recurrent and 
$\lim\sup X(t) = \infty$ while $\lim\inf X(t) = -\infty$
(Asmussen~\cite[Page 314, Proposition 2.10]{asmus03}).

Things are slightly different for the regulated process
$\{Y(t)\}$.  
If $\mu <0$, the process repeatedly
alternates between intervals of time where $Y(t)>0$ and intervals
where $Y(t)=0$.   As we find in Section~\ref{s:factorisation}, the
length of each cycle has finite expectation, the regulated process
is positive recurrent, and we may determine its stationary
distribution.  If $\mu >0$, then $Y(t)$ might not return to level 0
and so the process is transient.  Not surprisingly, it is
null-recurrent if $\mu = 0$.

We define $\bphi_n = \varphi(\theta_n)$, $n \geq 1$.
The process $\{\bphi_n\}$ of the phases visited at epochs of
regeneration is a Markov chain on the state space $\Sd$ by definition,
and its transition matrix is $H$, with
\[
H_{ij} = \P[\theta_{n+1} - \theta_n < \infty, 
\bphi_{n+1} =j| \bphi_n=i ], \qquad \mbox{$i$, $j$ in $\Sd$},.
\]
To determine $H$, we split the interval $(\theta_n,
\theta_{n+1}]$ in two and we condition  on the phase
occupied at time $\delta_n$.  Thus, $H$ is the product
\begin{equation}
   \label{e:H}
H =\Phi \, \Psi,
\end{equation}
where 
\begin{align}
  \nonumber
\Phi_{ik} & = \P[\delta_n - \theta_n < \infty, \varphi(\delta_n)=k | \bphi_n=i ], && i \in \Sd,
            k \in \Su, \\
   \label{e:defpsi}
\Psi_{kj} & = \P[\theta_{n+1}- \delta_n < \infty, \bphi_{n+1} = j | \varphi(\delta_n)=k], && k \in
            \Su, j \in \Sd.
\end{align}
The matrix $\Phi$ is easily determined:  $\varphi(t)$ remains in $\Sd \cup
\Sz$ during the interval $(\theta_n, \delta_n)$, and thus
\begin{equation}
   \label{e:Phi}
\Phi = \vligne{I & 0} 
  \left(-\vligne{Q_{\d\d}& Q_{\d0} \\  Q_{0\d} & Q_{00} }^{-1}\right)
  \vligne{Q_{\d\u} \\  Q_{0\u}}
\end{equation}
(Latouche and Ramaswami~\cite[Section 5.5]{lr99}).
To determine the matrix $\Psi$  requires more effort, and we devote Section~\ref{s:passage} to the
determination of its characteristic equation.    Once $\Psi$ is known,
the vector $\vrho$ in (\ref{e:G}) is determined by the system $\vrho\tr H =
\vrho\tr$, $\vrho\tr \vone = 1$.

In a similar manner, we decompose $M(x)$ as $M(x)= M(0)+M(0,x]$, where
$M(0)$ is the expected time spent at level 0 between the epochs
$\theta_n$ and $\delta_n$,  and $M(0,x]$ is the expected time spent in the
semi-open interval $(0,x]$.   We decompose $M(0)$ as
\[
M(0) = \vligne{M(0)_{\d\u} & M(0)_{\d\d} & M(0)_{\d 0} },
\]
and immediately note that $M(0)_{\d\u} =0$ as the fluid queue does not spend any time
at level 0 in a phase of $\Su$.  Furthermore,
\begin{equation}
   \label{e:mo}
  \vligne{ M(0)_{\d\d} & M(0)_{\d 0} } = \vligne{I & 0} 
  \left(-\vligne{Q_{\d\d}& Q_{\d0} \\  Q_{0\d} & Q_{00} }^{-1}\right),
\end{equation}
see \cite[Section 5.5]{lr99}.
The second term in $M(x)$ is equal to
\begin{equation}
   \label{e:phim}
M(0,x] = \Phi \, \wM(x),
\end{equation}
where $\widetilde M(x)$ is the matrix of expected times spent in
$(0,x]$ during the interval of time $(\delta_n, \theta_{n+1}]$.  It is
determined in Section~\ref{s:crossings}.

\section{First passage probabilities}
\label{s:passage}

We deal in this section with the fluid flow model $\{X(t),\varphi(t)\}$
without boundary.  Its transition structure is independent of the
level and so we shall not always pay close attention to the exact
value of $X$, but be more interested in differences of level.  For
instance, the matrix $\Psi$ defined in (\ref{e:defpsi}) might have
been defined as 
\[
\Psi_{ij} = \P[\Theta < \infty, \varphi(\Theta)=j | \varphi(0)=i],
\qquad i \in \Su, j \in \Sd.
\]
independently of $X(0)$, where $\Theta
= \inf\{t > 0: X(t) = X(0)\}$ is the first return time to the initial
level, starting in a phase of $\Su$.

Furthermore, let 
\begin{equation}
   \label{e:tau}
\tau^\d_x = \inf\{t : X(t) < X(0)-x\}
\end{equation}
be the first passage time to level $X(0)-x$, for $x >0$, and denote by $\gamma^\d(x) = \varphi(\tau^\d_x)$
the value of the phase when $X(0)-x$ is reached for the first
time.%
\footnote{As the trajectory of $X(t)$ is continuous, we might have
  defined $\tau^\d_x = \inf\{t : X(t) = X(0)-x\}$, but the strict
  inequality in (\ref{e:tau}) will be useful in Section~\ref{s:escape}.
}  
For $\varphi(0)$ in $\Sd$, the process $\{\gamma^\d(x):x \geq 0\}$ is a
continuous-parameter Markov process on the state space $\Sd$, and
there exists a matrix $U$ such that 
\[
(e^{Ux})_{ij} = \P[\tau_x^\d < \infty, \gamma^\d(x)=
j|\gamma^\d(0)=i],  \qquad i,j \in \Sd. 
\]
As $Q$ is irreducible, $\Psi$ is
strictly positive, meaning that $\Psi_{ij} >0$ for all $i$ in $\Su$,
$j$ in $\Sd$.  To see this, imagine a trajectory of
positive probability such that, starting from $(0,i)$, the process
returns at 0 for the first time in phase $j$, in finite time.  A
formal proof is in  Govorun {\it et al.}~\cite[Lemma~4.3]{glr11b}, or
Guo~\cite[Theorem~5]{guo02b}.   In consequence, the
off-diagonal elements of the generator $U$ are all strictly positive.
We discuss at greater length the algebraic properties of $U$ and
$\Psi$ in Section~\ref{s:factorisation}, but mention here the
most important ones, in relation to the stationary drift $\mu$:
\begin{itemize}
\item if $\mu \leq 0$, then $\Psi \vone = \vone$ and $U \vone =
  \vzero$, that is, $e^{Ux}$ is stochastic,
\item if $\mu >0$, then $\Psi \vone < \vone$ and $U \vone <
  \vzero$, that is, $e^{Ux}$ is  substochastic.%
\footnote{With $\vu$ and $\vv$ two vectors on the same set of indices,
  we write $\vu < \vv$ if $u_i < v_i$ for all index $i$.}
\end{itemize}

Finally, we define the matrix 
\begin{equation}
   \label{e:t}
T = \vligne{Q_{\u\u} & Q_{\u\d} \\ Q_{\d\u} & Q_{\d\d}}
 + \vligne{Q_{\u0} \\ Q_{\d0}}   (- Q_{00} )^{-1}   \vligne{Q_{0\u}&  Q_{0\d}}
\end{equation}
indexed by the states in $\Su \cup\Sd$ and we partition it as
\[
T= \vligne{T_{\u\u} & T_{\u\d} \\ T_{\d\u} & T_{\d\d}}.
\]
This is the generator of the censored process $\{\varphi(t)\}$
observed  only during the intervals of time spent in $\Su \cup \Sd$.

\begin{thm}
   \label{t:psiNu}
The matrix $U$ is given by
\begin{equation}
   \label{e:u}
U = |C_\d|^{-1} T_{\d\d}  + |C_\d|^{-1} T_{\d\u} \Psi,
\end{equation}
and $\Psi$, as defined in (\ref{e:defpsi}), is the minimal nonnegative
solution of the quadratic Riccati equation
\begin{equation}
   \label{e:riccati}
C_\u^{-1} T_{\u\d}  + C_\u^{-1} T_{\u\u} \Psi + \Psi |C_\d|^{-1}
T_{\d\d}  + \Psi  |C_\d|^{-1} T_{\d\u}  \Psi =0, 
\end{equation}
where $|C_\d|$ is the matrix of absolute values of the elements of $C_\d$.
\qed
\end{thm}

We give here a high-level justification, a formal proof is in Ahn and
Ramaswami~\cite{ar03} and da Silva Soares and Latouche~\cite{dssl03}.
To determine $\Psi$, it is simpler to use the censored process on
$\Su \cup \Sd$ since there is no change in the level while
$\varphi$ is in $\Sz$; that is why we use in (\ref{e:riccati}) the
generator $T$ of (\ref{e:t}) instead of the generator $Q$ of (\ref{e:q}).

Furthermore, instead of tracking the evolution of the phase process in
{\em time} as one might be tempted to do, we
track its  evolution over changes of the {\em level}.  The parameters
$|c_i|$, for $i \in \Su \cup \Sd$, are conversion rates of time to
fluid and their reciprocal $|c_i|^{-1}$ are conversion rates of fluid
to time, so that the matrix $|C|^{-1} T$ indicates how the phase evolves
as the fluid level is increasing or decreasing. 

With this in mind, we write 
\begin{equation}
   \label{e:psia}
\Psi = \int_0^\infty   e^{C_\u^{-1} T_{\u\u} y} \, C_\u^{-1} T_{\u\d} \,
e^{Uy} \, \ud y.
\end{equation}
The justification goes as follows (see the illustration in
Figure~\ref{f:riccati}).  The process starts in a phase of $\Su$, and
we assume without loss of generality that $X(0)=0$.  The factor
$e^{C_\u^{-1} T_{\u\u} y}$ in the right-hand side of (\ref{e:psia}) is
the probability that the phase remains in $\Su$ until the fluid has
increased up to level $y$.  The factor $C_\u^{-1} T_{\u\d} \ud y$ is
the probability that between the levels $y$ and $y+ \ud y$, the phase
moves to $\Sd$ and starts to decrease.  The factor $e^{Uy}$ is the
probability that the fluid eventually goes down to level 0.
Level $y$ is reached at some unspecified moment $t(y)$, and the
process moves without constraint between $t(y)$ and $\Theta$.

\begin{figure}[t]
\centering{
\includegraphics{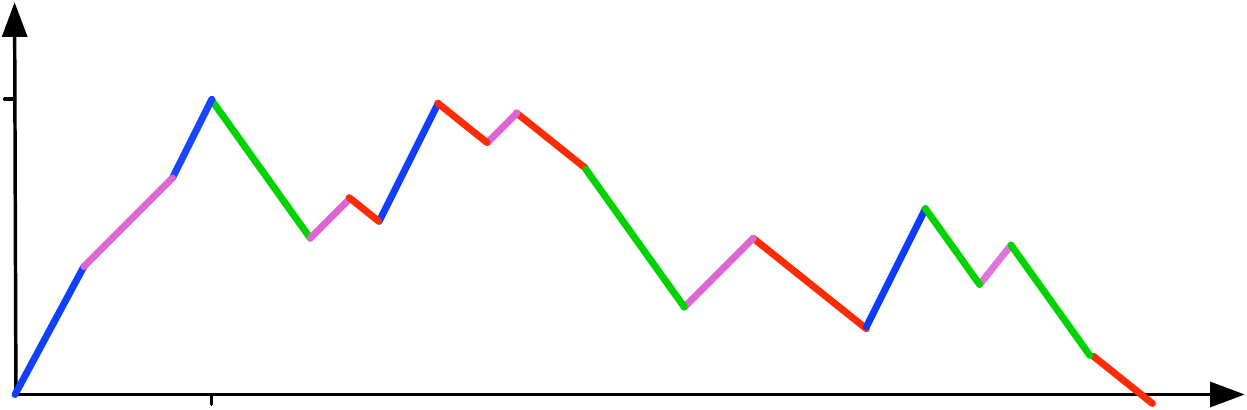} 
\put(-10,-10){\makebox(0,0)[c]{$t$}}
\put(-35,-10){\makebox(0,0)[c]{$\Theta$}}
\put(-300,-10){\makebox(0,0)[c]{$t(y)$}}
\put(-375,110){\makebox(0,0)[c]{$X(t)$}}
\put(-365,90){\makebox(0,0)[c]{$y$}}
\put(-365,5){\makebox(0,0)[c]{$0$}}
}
\caption{The fluid increases and reaches $y$ at some unspecified time
  $t(y)$.  Between $t(y)$ and $\Theta$ the fluid moves without constraint.}
\hrulefill
\label{f:riccati}
\end{figure}
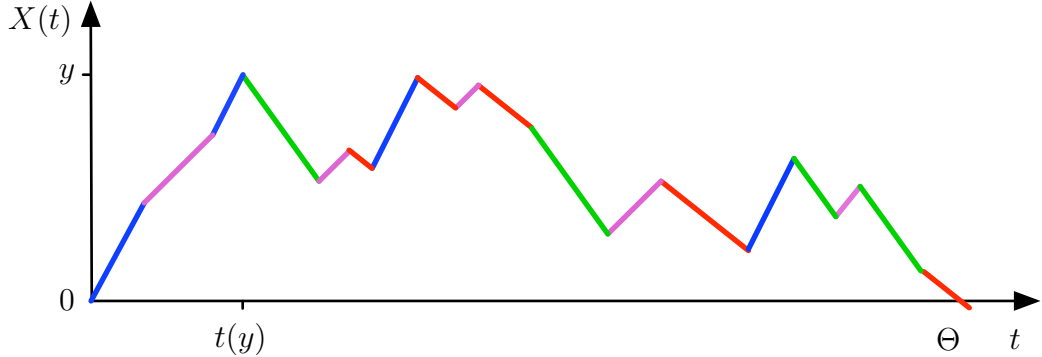

We pre-multiply both sides of (\ref{e:psia}) by $C_\u^{-1} T_{\u\u}$
integrate by part, and find
\begin{equation}
   \label{e:psib}
C_\u^{-1} T_{\u\d}  +     C_\u^{-1} T_{\u\u} \Psi  + \Psi U = 0.
\end{equation}
This is a nonsingular Sylvester equation (Lancaster and
Tismenetsky~\cite{lt85}) and we may characterise $\Psi$ as the
unique solution of 
(\ref{e:psib}) if $U$ is known.

To prove  (\ref{e:u}), we write
\begin{align}
  \nonumber
U_{ij} h & =
\P[\gamma^\d(u+h)= j | \gamma^\d(u)=i] + o(h), \\
   \label{e:upsi}
  & = (|C_\d|^{-1} T_{\d\d})_{ij} h + (|C_\d|^{-1} T_{\d\u} \Psi)_{ij} h + o(h)
\end{align}
for $i$, $j$ in $\Sd$.  The first term is the probability that
$\varphi$ changes from $i$ to $j$, the second term is the probability that
$\varphi$ changes from $i$ to a phase in $\Su$ at some level $u$ and is in phase $j$
when the process later returns to level $u$.  We illustrate this in
Figure~\ref{f:uNpsi}, where we plot the trajectory of the process from
level $y$ to level 0, and draw with solid lines the part that corresponds to
$\gamma^\d(x)$.  At level $y_a$ there is a simple change from phase 2
to phase 1; at level $y_b$ there is a change from phase 2 to phase 4
in $\Su$ and upon return to level $y_b$ the process is in phase 1.
%

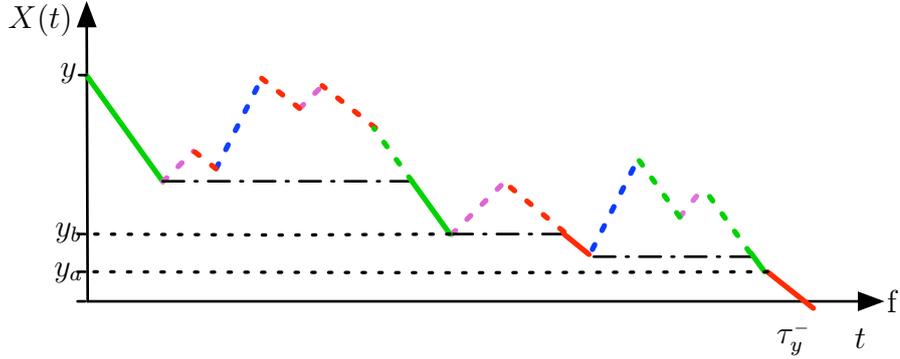
\begin{figure}[t]
\centering{
\includegraphics{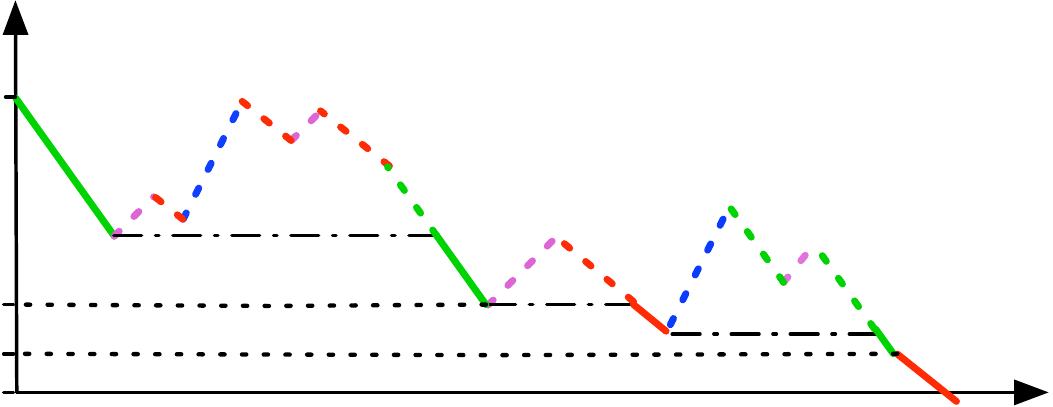} 
\put(-10,-10){\makebox(0,0)[c]{$t$}}
\put(-35,-10){\makebox(0,0)[c]{$\tau^\d_y$}}
\put(-317,110){\makebox(0,0)[c]{$X(t)$}}f
\put(-310,90){\makebox(0,0)[c]{$y$}}
\put(-310,30){\makebox(0,0)[c]{$y_b$}}
\put(-310,15){\makebox(0,0)[c]{$y_a$}}
}
\caption{The sample path of $X(t)$ from $y$ to 0 is reproduced
  from Figure~\ref{f:riccati}, the path of
  $\gamma^\d(x)$ corresponds to the solid downward segments, $y_a$ and
  $y_b$ are two of the levels where the  value of $\gamma^\d(x)$ changes.}
\hrulefill
\label{f:uNpsi}
\end{figure}

The Riccati equation (\ref{e:riccati}) is obtained by replacing in (\ref{e:psib})
$U$ by its expression in (\ref{e:u}).  
Over the years, very efficient algorithms have
been developed to solve~(\ref{e:riccati}), we describe two of these in
Section~\ref{s:algorithms}. 

Two other matrices, $\whU$ and $\whpsi$, may be defined at this stage.
Let $\tau^\u_x = \inf\{t : X(t) > X(0)+x\}$ be the first passage time
to level $X(0)+x$, for $x >0$, and denote by
$\gamma^\u(x) = \varphi(\tau^\u_x)$ the value of the phase when level
$X(0)+x$ is reached for the first time.  The matrix $\whU$ is the
generator of $\{\gamma^\u(x)\}$ and $\whpsi$ is the matrix of first
passage probability back to the initial level, given that the phase at
time 0 is in $\Sd$:
\[
\whpsi_{ij} = \P[\Theta < \infty, \varphi(\Theta)=j | \varphi(0)=i],
\qquad i \in \Sd, j \in \Su.
\]
We easily adapt the argument in Theorem \ref{t:psiNu} to prove the following.

\begin{cor}
   \label{t:psiNust}
The matrix $\whU$ is given by
\begin{equation}
   \label{e:ust}
\whU = C_\u^{-1} T_{\u\u}  + C_\u^{-1} T_{\u\d} \whpsi 
\end{equation}
and $\whpsi$ is the minimal nonnegative solution of the equation
\begin{equation}
   \label{e:riccatist}
|C_\d|^{-1} T_{\d\u} + |C_\d|^{-1} T_{\d\d}  \whpsi +
\whpsi C_\u^{-1} T_{\u\u}  + \whpsi C_\u^{-1} T_{\u\d}    \whpsi =0.
\end{equation}
Furthermore,  $\whpsi \vone < \vone$ and  $\whU \vone <
  \vzero$ if $\mu < 0$, while $\whpsi\vone = \vone$ and $\whU \vone =
  \vzero$ if $\mu \geq 0$.
\qed
\end{cor}

\section{Number of crossings}
\label{s:crossings}

As a preliminary step to determining the expected sojourn time
$\wM(x)$ in $(0,x]$ during an interval $(\delta_n, \theta_{n+1}]$, we
analyse the number of times the fluid crosses a given level during a
regenerative interval. We define $N_{ij}(x)$ to be the expected number
of times $(X,\varphi) = (X(0)+x,j)$ during the interval $(0, \Theta]$,
given that $\varphi(0)=i$, for $i$ and $j$ in $\Su$.

\begin{thm}
   \label{t:K}
The matrix $N(x)$ is given by $N(x)=e^{Kx}$,  where 
\begin{equation}
   \label{e:K}
K = C_\u^{-1} T_{\u\u}  + \Psi |C_\d|^{-1} T_{\d\u}
\end{equation}
is a matrix indexed by $\Su$.
\qed
\end{thm}
The formal proof is given in Ramaswami~\cite{ram99} and proceeds as
follows.%
\footnote{The argument is similar to the one  used in Latouche and
  Ramaswami~\cite[Theorem 6.2.7]{lr99} in the context of
  Quasi-Birth-and-Death processes.}  
Assume without loss of generality
that $X(0)=0$, take $x$ and $y >0$, and count the expected number of
visits to $(x+y,j)$, starting from $(0,i)$, before the first return to
level 0.  We group the visits to $(x+y,j)$ into subintervals between
successive up-crossings of level $x$, and write that
\[
N_{ij}(x+y) = \sum_{k \in \Su} N_{ik}(x) N_{kj}(x, x+y), \qquad i, j \in \Su,
\]
where $N_{kj}(x, x+y)$ is the expected number of visits to $(x+y,j)$
between two successive visits to level $x$, starting from $(x,k)$.
Remember that it is the {\em distance} $y$ between the target $x+y$
and the starting level $x$ matters, not the specific location of the
latter.  Thus, $N(x, x+y) = N(0,y)$ and so the equation above is also
written as $N(x+y)=N(x)N(y)$.  From the semi-group property, we
conclude that there exists a matrix $K$ such that $N(x)= e^{Kx}$.

Next, we approximate $(e^{Kh})_{ij}$ for $h$ small and $i \not= j$ as
\begin{align*}
(e^{Kh})_{ij}  & = K_{ij}h + o(h) 
\\
 & = c_i^{-1} Q_{ij} h + \sum_{k \in \Sz} c_i^{-1} Q_{ik}
   ((-Q_{00})^{-1} Q_{0\u})_{kj} h
\\
 & \quad + (1+c_i^{-1} Q_{ii} h) \sum_{\ell \in \Sd} \Psi_{i \ell}
  |c_\ell|^{-1} Q_{\ell j}  h   
\\
 & \quad + (1+c_i^{-1} Q_{ii} h) \sum_{\ell \in \Sd} \Psi_{i \ell}
  |c_\ell|^{-1} \sum_{k \in \Sz}  |c_\ell|^{-1}  Q_{\ell k}
   ((-Q_{00})^{-1} Q_{0\u})_{kj} h   + o(h).
\end{align*}
The first two terms are about the process being in phase $j$ at the
first crossing of level $h$;
\begin{itemize}
\item the first  is the probability that during the interval of time
$c_i^{-1} h$ the phase process changes from $i$ to $j$ and is still in
$j$ when the fluid crosses level $h$;
\item the second term is the
probability that during 
that interval of time, the phase switches to $k$ in $\Sz$ at some
unspecified level below $h$, remains at
that level until there is a jump to $j$, and is still in
$j$ when the fluid crosses level $h$.
\end{itemize}
The next two terms cover the circumstances where the process does not
leave phase $i$ before crossing level $h$, returns to level $h$ in
some phase $\ell$ in $\Sd$, switches to phase $j$ during the interval
of time $|c_\ell|^{-1} h$ and eventually crosses level $h$ in phase
$j$.  The $o(h)$ term captures the negligible probabilities of
crossing level $h$ more than once.  Simple manipulations give us
\begin{align*}
(e^{Kh})_{ij} & = (c_i^{-1} T_{ij}  + \sum_{\ell \in \Sd} \Psi_{i \ell}
   |c_\ell|^{-1} T_{\ell j} )h +o(h), \qquad  i\not = j \in \Su.
\end{align*}
A similar argument holds for $i=j$ and so is (\ref{e:K}) justified.

We partition the matrix $\wM(x)$ in three blocks, corresponding to the
three subsets $\Su$, $\Sd$ and $\Sz$ of phases, and we deal with each
one separately.

\begin{thm}
   \label{t:mofx}
The matrix of expected sojourn times in $(0,x]$ during a regenerative
interval is
\[
\wM(x) =
\vligne{\wM_\u(x) & \wM_\d(x) & \wM_0(x)}
\]
with 
\begin{align}
   \label{e:mu}
\wM_\u(x)  & = \FF(K,x) C_\u^{-1},
\\
   \label{e:md}
\wM_\d(x) & =  \FF(K,x)\Psi |C_\d|^{-1},
\\
   \label{e:mox}
\wM_0(x) & = \FF(K,x) \vligne{C_\u^{-1} & \Psi |C_\d|^{-1} }
\vligne{Q_{\u0} \\ Q_{\d0}}   (- Q_{00} )^{-1}.
\end{align}
where $\FF(K,x) = \int_0^x e^{Ku} \, \ud u$.
\end{thm}

\begin{proof}
A formal proof of (\ref{e:mu}, \ref{e:md}) is given in Latouche and
Nguyen~\cite{ln15b}, we give here a heuristic argument, and we give a
justification for (\ref{e:mox}).

For $i$, $j$ in $\Su$, we argue that
\[
\wM_{ij} (x) = \int_0^x N_{ij}(u) c_j^{-1} \, \ud u.
\]
To see this, we interpret $c_j^{-1 }\ud u$ as the expected time spent
by $X$ in $(u, u+ \ud u)$ each time there is a visit to $(u,j)$, and
multiply this by the expected number $N_{ij}(u)$ of such visits, so
that $N_{ij}(u) c_j^{-1} \, \ud u$ is the total {\em time} spent by
the process.  There only remains for us to integrate over the interval
$(0,x)$ and use Theorem~\ref{t:K}.

Next, we define the matrix $N'(x)$ of expected number of
down-crossings of level $x$ in a phase of $\Sd$.  As $X(0)=0 < x$, of
necessity each down-crossing of level $x$ is preceded by an
up-crossing of that same level.  So, if we condition on the phase at
the up-crossing, we find that 
\[
N'_{ij}(x) = \sum_{k \in \Su} N_{ik}(x) \Psi_{kj}, \qquad i \in \Su, j
\in \Sd
\] 
and (\ref{e:md}) is proved by repeating the argument of (\ref{e:mu}).

Finally, let us use the expression {\em excursion to $\Sz$} for
intervals of time spent in $(0,x) \times \Sz$ and separated by visits
to $(0,x) \times \Su\cup\Sd$.  Equation 
(\ref{e:mox}) expresses $\wM_0(x)$ as the product of the
expected number of excursions by the time spent in
individual phases during excursion.  Indeed,
$Q_{\u0}$ and $Q_{\d0}$ are the matrices of transition rates
from $\Su$ and $\Sd$ to $\Sz$, and 
$\FF(K,x) \vligne{C_\u^{-1} & \Psi
  |C_\d|^{-1} }$ is the matrix of expected time spent in the states of 
$(0,x) \times \Su\cup\Sd$.  Thus, the component $N''_{ik}(x)$ of the
product 
\[
N''(x) = \FF(K,x) \vligne{C_\u^{-1} & \Psi |C_\d|^{-1} }
\vligne{Q_{\u0} \\ Q_{\d0}}
\] 
is the expected number of excursions that start in phase $k$.  As
$(-Q_{00}^{-1})_{kj}$ is the expected time spent in $j$ during an
excursion which starts in $k$,  this proves~(\ref{e:mox}).
\end{proof}

As we shall see in Section~\ref{s:factorisation}, all the eigenvalues
of $K$ are in $\CC_{\leq 0}$, that is, they  have a negative real
part.  If  $\mu < 0$, the real parts are all strictly negative,
 otherwise $K$ has one eigenvalue equal to 0, the
others being in $\CC_{<0}$.  With this, we may express the integral 
$\int_0^x e^{Ku} \, \ud u$ as follows.
\begin{itemize}
\item
   If $\mu < 0$, then 
\begin{equation}
   \label{e:muneg}
\FF(K,x) = (-K^{-1})(I - e^{Kx}),
\end{equation}
\item
  if $\mu \geq 0$, then 
\[
\FF(K,x) = (-K^\#) (I - e^{Kx}) + x \vv
  \vu\tr,
\]
 where $\vv$ and $\vu$ are the normalised right- and left-eigenvectors
  of $K$ for the eigenvalue 0, and $K^\#$ is the group inverse%
\footnote{The group inverse $K^\#$ of $K$ is the unique matrix such
  that $K^\# K = I - \vv \vu\tr$, $K^\# \vv = \vzero$.   See Campbell
  and Meyer~\cite[Chapter~7]{cm91}.} 
 of $K$.
\end{itemize}

In the same manner as we defined $\whU$ and $\whpsi$, we may define
$N^*(x)$ as the number of crossings of level $X(0)-x$ in a phase of $\Sd$,
starting from a phase in $\Sd$, before the first return to level
$X(0)$, and we have the following corollary to Theorem \ref{t:K}:

\begin{cor}
The matrix $N^*(x)$ is given by $N^*(x)=e^{\whK x}$,  where 
\begin{equation}
   \label{e:Kst}
\whK = |C_\d|^{-1} T_{\d\d}  + \whpsi C_\u^{-1} T_{\u\d}
\end{equation}
is  indexed by $\Sd$.
Furthermore, the eigenvalues of $\whK$ are in $\CC_{<0}$, with the
exception of one eigenvalue equal to 0 if $\mu \leq 0$.
\qed
\end{cor}

\section{Stationary distribution}
\label{s:stationary}

We collect the results obtained in the preceding sections and we express
as follows the stationary distribution of the fluid queue, when it exists.

\begin{thm}
   \label{t:stationary}
If $\mu < 0$, the regulated process $(Y(t), \varphi(t))$ has a
stationary distribution, given by
\[
\vG\tr(x) = c ( \vligne{0 & \vp\tr_\d & \vp\tr_0}
+ (\vp\tr_\d Q_{\d\u} + \vp\tr_0 Q_{0 \u}) \wM(x)),
\]
where 
\begin{align}
\nonumber
\vligne{\vp\tr_\d & \vp\tr_0} & = \vligne{\vrho\tr & 0} \left(-\vligne{Q_{\d\d}&
                                                                 Q_{\d0}
  \\  Q_{0\d} & Q_{00} }^{-1} \right) ,
\\[0.25\baselineskip]
   \label{e:mtilde}
 \wM(x) & = (-K)^{-1} (I - e^{Kx} ) \vligne{C_\u^{-1} & \Psi |C_\d|^{-1}
                                           & \Gamma},
\\[0.25\baselineskip]
\nonumber
 \Gamma & = (C_\u^{-1}   Q_{\u0}  +    \Psi |C_\d|^{-1} Q_{\d0})   (- Q_{00} )^{-1},
\end{align}
and $c$ is the normalising constant, with
\[
c^{-1} = \vligne{ \vp\tr_\d & \vp\tr_0} (\vone + \vligne{Q_{\d\u} \\ Q_{0
    \u}} (-K)^{-1} (C_\u^{-1} \vone + \Psi |C_\d|^{-1} \vone + \Gamma\vone)).
\]
\end{thm}

\debproof The expression (\ref{e:mtilde}) for $\wM(x)$ results from
Theorem~\ref{t:mofx} and Equation~(\ref{e:muneg}).  Furthermore, the
vector $\vm$ of conditional expected duration of of a regeneration
interval, given the initial phase, is given by
\begin{align}
  \nonumber 
\vm & = \lim_{x \rightarrow \infty} M(x) \vone 
\\
  \nonumber 
  & = M(0) \vone + \lim_{x \rightarrow \infty} \Phi \wM(x) \vone
    \qquad \mbox{by (\ref{e:mo}, \ref{e:phim}),}
\\
   \label{e:smallm}
  & = M(0) \vone + \Phi (-K)^{-1} (C_\u^{-1} \vone + \Psi |C_\d|^{-1}
    \vone +
    \Gamma \vone)
\end{align}
by (\ref{e:mtilde}).  The normalising constant $c$ equals $(\vrho\tr
\vm)^{-1}$ by (\ref{e:G}); simple calculations complete the proof.
\finproof

The stationary distribution appears under various forms in the literature
(Asmussen~\cite{asmus95b}, Govorun {\it et al.}~\cite{glr11b},
Ramaswami~\cite{ram99}, Rogers~\cite{roger94}); the
matrix $\wM(x)$ is a common feature in papers that rely on
matrix-analytic methods, the vector of probability mass at
zero and the left-factor of $\wM(x)$ are given very different
expressions, depending on the specific approach followed by the authors. 

\section{Wiener-Hopf factorisation}
\label{s:factorisation}

The matrices $U$, $\Psi$ and $K$ defined in Sections \ref{s:passage}
and \ref{s:crossings} are related in many ways.  Obviously, $U$ and
$K$ are determined by $\Psi$ through (\ref{e:u}) and (\ref{e:K}),
respectively, but $\Psi$ may be seen as a function of $U$ by
(\ref{e:psib}) and we might combine (\ref{e:riccati}) and (\ref{e:K})
to write
\[
C_\u^{-1} T_{\u\d}  + K \Psi + \Psi |C_\d|^{-1}
T_{\d\d}  =0, 
\]
from which we conclude that $\Psi$ is a function of $K$.  There exist
also numerous relations between the 3-tuples $(U, \Psi, K)$ and $(\whU,
\whpsi, \whK)$, as we briefly discuss below.

The starting point is the Wiener-Hopf factorisation 
\begin{equation}
   \label{e:whu}
C^{-1} T \begin{bmatrix}  I & \Psi \\ \whpsi & I \end{bmatrix}
=
\begin{bmatrix}  I & \Psi \\ \whpsi & I \end{bmatrix}
\begin{bmatrix} \whU & 0 \\ 0 & -U\end{bmatrix}
\end{equation}
proved in Rogers~\cite{roger94}, with
\[
C = \vligne{C_\u \\ & C_\d}.
\]
Equation (\ref{e:whu}) may be proved by direct
verification, starting from the Riccati equations (\ref{e:riccati})
and (\ref{e:riccatist}), and the expressions (\ref{e:u}) and
(\ref{e:ust}) for $U$ and $\whU$.   The equation below may also be
proved by direct verification:
\begin{equation}
   \label{e:whk}
\begin{bmatrix}  I & -\Psi \\ -\whpsi & I \end{bmatrix}
C^{-1} T.
=
\begin{bmatrix} K & 0 \\ 0 & -\whK\end{bmatrix} 
\begin{bmatrix}  I & -\Psi \\ -\whpsi & I \end{bmatrix},
\end{equation}
An immediate consequence is that $U$ and $\whU$ on the one hand, $K$
and $\whK$ on the other hand, share the eigenvalues of $C^{-1} T$.  We
denote by $m_\u$ and $m_\d$ respectively the number of phases in $\Su$
and $\Sd$, and we denote by $\{\lambda_i: 1\leq i \leq m_\u + m_\d\}$
the eigenvalues of $C^{-1} T$, labeled in increasing value of their
real part. 

\begin{thm}
   \label{t:roots}
The eigenvalues $\lambda_{{m_\u}}$ and $\lambda_{{m_\u}+1}$ are real, and are distinct from the others: 
\begin{equation}
   \label{e:separate}
\re(\lambda_1) \leq \cdots \leq
\re(\lambda_{{m_\u}-1})  < \lambda_{m_\u}  \leq \lambda_{{m_\u}+1} < \re(\lambda_{{m_\u}+2}).
 \leq \cdots \leq \re(\lambda_m).
\end{equation}
Furthermore, $\lambda_i$, $1 \leq i \leq m_\u$ are  the eigenvalues
of $K$ and also the eigenvalues of $\whU$, while $\lambda_{{m_\u}+1}$,
\ldots $\lambda_{m_\u+m_\d}$ are those of $-U$ and of $-\whK$.
Finally,
\begin{itemize}
\item if $\mu<0$, then $\lambda_{m_\u} < 0 = \lambda_{{m_\u}+1}$,
\item if $\mu = 0$, then $\lambda_{{m_\u}} = 0 = \lambda_{{m_\u}+1}$,
\item if $\mu >0$, then $\lambda_{{m_\u}} = 0 < \lambda_{{m_\u}+1}$.
\end{itemize}

\end{thm}
Details are given in \cite{glr11b}.  The main properties are
summarised in Table~\ref{t:eigenvalues}.

\begin{table}[t]
\centering{
\begin{tabular}{r|rrr}
drift $\mu$ & $\,< 0$ & \quad 0 & $>0$ \\
\hline
$\Psi \vone$ & $\vone$ & $\vone$ & $< \vone$ \\
$\whpsi\vone$ & $< \vone$ & $\vone$ & $ \vone$ \\
 $U$, $\whK$ & 0 & 0 & $<0$ \\
$\whU$, $K$ & $<0$ & 0 & 0 \\
\end{tabular}
}
\caption{Crossing probabilities, and maximal eigenvalues of $U$, $K$,
  $\whU$, and $\whK$ in relation to the
  asymptotic drift $\mu$.}
\hrulefill
   \label{t:eigenvalues}
\end{table}

Actually, $K$ and $\whU$ are similar matrices, and $\whK$ is similar
to $U$.  This is easy to show if  the stationary drift is different
from 0 for, we find from (\ref{e:whu}, \ref{e:whk})
after some simple algebraic manipulation that
\begin{equation}
   \label{e:kNust}
K (I-\Psi \whpsi) = (I-\Psi \whpsi) \whU 
\quad \mbox{and} \quad
 \whK (I-\whpsi \Psi) = (I-\whpsi \Psi) U.
\end{equation}
If $\mu \not= 0$, then both $\Psi \whpsi$ and $\whpsi \Psi$ are
strictly sub-stochastic matrices, $I- \Psi \whpsi$ and
$I-\whpsi \Psi$ are non-singular,  and  we may write 
\begin{align*}
\whU   = (I-\Psi \whpsi)^{-1}  K (I-\Psi \whpsi) 
\\
\whK = (I-\whpsi \Psi) U (I-\whpsi \Psi)^{-1}
\end{align*}
which shows that $K$ is similar to $\whU$ and
$\whK$ is similar to $U$.

If $\mu=0$, however, $\Psi \whpsi \vone = \vone$ and $\whpsi \Psi
\vone = \vone$, 
$I - \Psi \whpsi$ and $I-\whpsi \Psi$ are both
singular, and the argument above fails.  Instead, one must
develop the Jordan chain argument from the proof of \cite[Lemma~4.6]{glr11b}.  

As a matter of fact, we often find that $\mu = 0$ is a case that
presents additional difficulties.  This will be seen in
Theorem~\ref{t:fullsolution} about escape probabilities --- that is
but one example. Moreover, the convergence of computational algorithms
is much slower (Guo~\cite{guo06}).

\section{First passage times}
\label{s:times}

The matrix $H$ analysed in Section~\ref{s:regenerative} gives us the
joint probability that a regeneration interval is finite, and the
phase at the end of the interval.  Here, we are interested in the {\em
  distribution} of the length of the regeneration interval, that is,
the distribution functions
\[
H_{ij}(t)  = \P[\theta_{n+1} - \theta_n \leq t, \bphi_{n+1} = j |
\bphi_n =i],
\]
for $i$, $j$ in $\Sd$.  Using the same partition that gave us
(\ref{e:H}), we write $H(t)$ as the convolution product $H(t) =
\Phi * \Psi(t)$, with 
\begin{align*}
\Phi_{ik}(t) & = \P[\delta_n - \theta_n \leq t, \varphi(\delta_n)=k | \bphi_n=i ], && i \in \Sd,
            k \in \Su, \\
\Psi_{kj}(t) & = \P[\theta_{n+1} - \delta_n \leq t, \bphi_{n+1} = j | \varphi(\delta_n)=k], && k \in
            \Su, j \in \Sd.
\end{align*}
As usual, it is easier to characterise the regenerative intervals
through their LS transforms
\begin{equation}
   \label{e:hs}
\whH_{ij}(s)  = \int_0^\infty e^{-st} \, \ud H_{ij}(t),
\end{equation}
and to write $\whH(s)= \whphi(s) \whpsi(s)$ where $\whphi(s)$ and $\whpsi(s)$
are the matrices of LS transforms of $\Phi(t)$ and $\Psi(t)$,
respectively.  It is well-known that LS transforms may be interpreted in
probabilistic terms through the introduction of an exponential random
variable $V$ with parameter $s$, independent of the fluid flow
process: we rewrite (\ref{e:hs}) as 
\[
\whH_{ij}(s) = \P[ \theta_{n+1} - \theta_n \leq V, \bphi_{n+1} =j|
\bphi_n=i ],
\]
and similarly
\begin{align*}
\whphi_{ik}(s) & = \P[\delta_n - \theta_n \leq V, \varphi(\delta_n)=k
                 | \bphi_n=i ], 
\\
\whpsi_{kj}(s) & = \P[\theta_{n+1} - \delta_n \leq V, \bphi_{n+1} = j | \varphi(\delta_n)=k].
\end{align*}
At this point, it is easy to verify that $\whphi(s)$ and
$\whpsi(s)$ are given by slight modifications of (\ref{e:Phi}) and
(\ref{e:riccati}):
\begin{equation}
   \label{e:phih}
\whphi(s) = \vligne{I & 0} 
  \left(-\vligne{Q_{\d\d} - sI & Q_{\d0} \\  Q_{0\d} & Q_{00}
      -sI}^{-1} \right)
  \vligne{Q_{\d\u} \\  Q_{0\u}},
\end{equation}
and $\whpsi(s)$ is the minimal nonnegative solution of 
\begin{equation}
   \label{e:psih}
C_\u^{-1} T_{\u\d}(s)  + C_\u^{-1} T_{\u\u}(s) \whpsi(s) + \whpsi(s) |C_\d|^{-1}
T_{\d\d}(s)  + \whpsi(s)  |C_\d|^{-1} T_{\d\u}(s)  \whpsi(s) =0,
\end{equation}
where
\[
T(s) = \vligne{Q_{\u\u}-sI & Q_{\u\d} \\ Q_{\d\u} & Q_{\d\d}-sI}
 + \vligne{Q_{\u0} \\ Q_{\d0}}   (sI- Q_{00} )^{-1}   \vligne{Q_{0\u}&  Q_{0\d}}.
\]
The Riccati equation
may be solved for any given $s$ by the same algorithms as discussed in
Section~\ref{s:algorithms}, and this makes it feasible to compute the
distributions themselves by numerical inversion procedures.

Moments of
first passage times are obtained by taking the derivatives of $\whH(s)
= \whphi(s) \whpsi(s)$ and evaluating it for $s=0$.  Derivatives of
$\whphi(s)$ are easily obtained from (\ref{e:phih}) but those of
$\whpsi(s)$ are more involved, as shown in Bean {\it et
  al.}~\cite{brt05}, and are expressed as solutions of nonsingular
Sylvester equations.  
The {\em first moment} may be obtained in a more straightforward manner,
as we show in (\ref{e:smallm}).


\section{Escape from an interval}
\label{s:escape}

Section~\ref{s:passage} is about the distribution of the phase
upon the first passage of $X(t)$ to a given level.  Here, we deal with
the first passage to the boundary of a finite interval: assuming that
$X(0)=0$, we look for the distribution of the phase when the
level escapes for the first time from the interval $(-a, b)$,
with $a$ and $b \geq 0$.

We define as follows the matrices  $A^{(a,b)}$ and $B^{(a,b)}$ indexed
by  $\Ss \times \Sd$ and $\Ss \times \Su$, respectively:
\begin{align*}
A_{ij}^{(a,b)} &= \P[\tau_a^\d < \tau_b^\u, \varphi(\tau_a^\d) =
j| X(0)=0, \varphi(0)=i],  \qquad \mbox{$i \in \Ss$, $j \in \Sd$,} 
\\
B_{ij}^{(a,b)} &= \P[\tau_b^\u < \tau_a^\d, \varphi(\tau_b^\u) =
j| X(0)=0, \varphi(0)=i],  \qquad \mbox{$i \in \Ss$, $j \in \Su$.}
\end{align*}
The matrices are partitioned into the usual subblocks:
\[
\vligne{ B^{(a,b)}  & A^{(a,b)}} =
\vligne{B_{\u\u}^{(a,b)} & A_{\u\d}^{(a,b)}  \\ B_{\d\u}^{(a,b)} &
  A_{\d\d}^{(a,b)}  \\ B_{0\u}^{(a,b)} & A_{0\d}^{(a,b)}}.
\]
If $\varphi(0)$ is in $\Sz$, the process remains at level 0 for a
while, before jumping to a phase in either $\Su$ or $\Sd$.
We condition on the first phase visited
either in $\Su$ or $\Sd$  and we find that
\begin{equation}
   \label{e:codo}
\vligne{B_{0\u}^{(a,b)} & A_{0\d}^{(a,b)}} = (-Q_{00})^{-1}
\vligne{Q_{0\u} & Q_{0\d}}
\vligne{B_{\u\u}^{(a,b)} & A_{\u\d}^{(a,b)}  \\ B_{\d\u}^{(a,b)} & A_{\d\d}^{(a,b)}  }.
\end{equation}

\begin{rem} \em
   \label{r:ineq}
We see here why we defined $\tau_x^\d$ in Section~\ref{s:passage} as
$\tau^\d_x = \inf\{t : X(t) < X(0)-x\}$.  If we had defined it as
$\tau^\d_x = \inf\{t : X(t) = X(0)-x\}$, then (\ref{e:codo}) would not
have held  for $a=0$.  Similarly, it would
not have held for $b=0$ if we had defined
$\tau^\u_x = \inf\{t : X(t) \geq X(0)+x\}$.
\end{rem}

We need the following lemma to determine the remaining entries of 
 $B^{(a,b)}$  and    $A^{(a,b)}$.

\begin{lem}
   \label{t:system}
   For $\varphi(0)$ in $\Su$ or $\Sd$, the escape probability matrices
   are solutions of the linear system
\begin{equation}
   \label{e:system}
\vligne{B_{\u\u}^{(a,b)} & A_{\u\d}^{(a,b)}  \\ B_{\d\u}^{(a,b)} &
  A_{\d\d}^{(a,b)}}
(I + \PP)
=
\UU,
\end{equation}
where 
\begin{align*}
\PP & = \vligne{0 & \Psi e^{U(a+b)} \\ \whpsi e^{\whU(a+b)} & 0},
\\[0.5\baselineskip]
\UU & = \vligne{e^{\whU b}   & \Psi e^{U a}  \\ 
         \whpsi e^{\whU b}  &  e^{Ua}}.
\end{align*}
\end{lem}
\begin{proof}
Assume $\varphi(0)$ is in $\Su$.  We have
\begin{equation}
   \label{e:ustb}
e^{\whU b} = B_{\u\u}^{(a,b)} + A_{\u\d}^{(a,b)} \whpsi e^{\whU(a+b)}.
\end{equation}
Indeed, the left-hand side gives the distribution of the phase when
the process has moved up from level 0 to level $b$, it is decomposed
in the right-hand side as the sum of the probability that the process reaches $b$
without going down to $-a$ and the probability that it goes down to
$-a$ first, then returns to level $-a$ from below, and eventually goes
up by $a+b$ units, from $-a$ to $b$. 
Similarly,
\begin{equation}
   \label{e:ua}
\Psi e^{U a} = A_{\u\d}^{(a,b)}  +  B_{\u\u}^{(a,b)} \Psi e^{U(a+b)}.
\end{equation}
Equations (\ref{e:ustb}, \ref{e:ua})  form  the first row of the system (\ref{e:system}).  The
argument for the second row is similar.
\end{proof}

If $\mu \not= 0$, then (\ref{e:system}) is nonsingular and it has a
unique solution.  The reason is that either $\Psi e^{U(a+b)}$ or
$ \whpsi e^{\whU (a+b)}$ is substochastic, and so the series
$\sum_{\nu \geq 0} (-1)^\nu \PP^\nu$ is converging to $(I+ \PP)^{-1}$
(see da Silva Soares and Latouche~\cite{dssl03} for details).  If
$\mu=0$, then both $\Psi e^{U(a+b)}$ and $ \whpsi e^{\whU (a+b)}$ are
stochastic matrices, $I+\PP$ is singular and we need to add one
equation.

Equation (\ref{e:supplement}) is one such choice, as we prove in
Theorem \ref{t:fullsolution}.  This equation is identical to the
one given for Markov-modulated {\em Brownian motion} in
Ivanovs~\cite{ivano10} in a comment after Theorem 3.1, referring to a
result obtained in D'Auria {\it et al.}~\cite[Section 7]{aikm10b} by a spectral
decomposition argument.  The proof given here is based on the analysis
of the stochastic process itself, it is new and for that reason we
give all technical details.

\begin{lem}
   \label{t:null}
If $\mu= 0$, then the escape probability matrix is such that
\begin{equation}
   \label{e:supplement}
\vligne{B_{\u\u}^{(a,b)} & A_{\u\d}^{(a,b)}  \\ B_{\d\u}^{(a,b)} &
  A_{\d\d}^{(a,b)}}
\vbeta
= 
\vligne{\vect h_\u \\ \vect h_\d}, 
\end{equation}
where $\vect h = -Q^\# \vcc$ and 
\begin{equation}
   \label{e:c}
\vbeta = \vligne{\ \ b \vone_\u + \vect h_\u  \\ -a \vone_\d + \vect h_\d}.
\end{equation}
\end{lem}
\begin{proof}
We know from Coolen-Schrijner and van Doorn~\cite[Section 3]{cv02}
  that the deviation matrix of $Q$ is equal to the group inverse of
  $-Q$, that is,  
$ \int_0^\infty (e^{Qu} - \vone \cdot \vect\alpha\tr) \, \ud u = -
Q^\# $.  
Therefore,
\begin{align}
   \nonumber
 \vect h & = \int_0^\infty (e^{Qu} - \vone \cdot \vect\alpha\tr) \,
  \ud
  u \ \vcc \\
   \nonumber
  & = \lim_{t \rightarrow \infty } \int_0^t (e^{Qu} - \vone \cdot
 \vect \alpha\tr) \, \ud   u \ \vcc \\
   \nonumber
 & =  \lim_{t \rightarrow \infty} \int_0^t (e^{Qu} \vcc
  - \vone \cdot \vect\alpha\tr \vcc) \, \ud u  \\
   \label{e:local}
  & = \lim_{t \rightarrow \infty} \int_0^t e^{Qu} \vcc \, \ud u
\end{align}
as $\vect\alpha\tr \vcc = \mu$ and $\mu = 0$ by assumption.
Furthermore, $(e^{Qu})_{ij}$ is the probability   $\P[\varphi(u)=j | \varphi(0)=i ]$ and so
(\ref{e:local}) may be interpreted by (\ref{e:x})  as 
\[
h_i = \lim_{t \rightarrow
  \infty} \E[X_t|X(0)=0, \varphi(0)=i],
\]
and we write $\vect h = \E_0[X_\infty|\varphi(0)]$ for short.
By conditioning on the first time the process escapes from $(-a, b)$, it is straightforward to verify that
\begin{align*}
\vect h & = \E_0[X_\infty \mathds{1}\{\tau_b^\u <
\tau_a^\d\}|\varphi(0)]  +  \E_0[X_\infty \mathds{1}\{\tau_a^\d <
\tau_b^\u\}|\varphi(0)] \\
 & = B^{(a,b)} (b \vone_\u + \vect h_\u) + A^{(a,b)}(-a \vone_\d + \vect h_\d)
\end{align*}
and this concludes the proof.
\end{proof}

\begin{rem}   \em
   \label{r:tbd}
The vector $\vh$ has a number of interesting properties that we need
later.  We observe that $\valpha\tr \vh =  0$ since $\valpha\tr
Q^\#=\vzero$.  Furthermore,
\begin{align}
   \label{e:hp}
\vh_\u & = \Psi \vh_\d  \quad \mbox{if\ } \mu \leq 0,
\\
   \label{e:hm}
\vh_\d & = \Psi \vh_\u \quad \mbox{if\ }  \mu \geq 0.
\end{align}
To justify the first equation, we use the interpretation given to
$\vh$, and we use the fact that if $\mu \leq 0$, starting from a phase
in $\Su$, the process returns to 0 in finite time with probability 1.
The justification of (\ref{e:hm}) is similar.
\end{rem}

In summary, the distribution of the phase upon escaping from the
interval $(-a,b)$ is given in the next theorem.

\begin{thm}
   \label{t:fullsolution}
The distribution of the phase at  first escape from $(-a,b)$ is given
by 
\[
\vligne{B_{0\u}^{(a,b)} & A_{0\d}^{(a,b)}} = (-Q_{00})^{-1}
\vligne{Q_{0\u} & Q_{0\d}}
\vligne{B_{\u\u}^{(a,b)} & A_{\u\d}^{(a,b)}  \\ B_{\d\u}^{(a,b)} & A_{\d\d}^{(a,b)}  }
\]
if $\varphi(0) \in \Sz$.
If $\mu \not= 0$ and $\varphi(0)$ is in $\Su$ or $\Sd$, then 
\begin{equation}
   \label{e:solutiona}
\vligne{B_{\u\u}^{(a,b)} & A_{\u\d}^{(a,b)}  \\ B_{\d\u}^{(a,b)} &
  A_{\d\d}^{(a,b)}}
=  \UU (I+\PP)^{-1} ,
\end{equation}
where $\UU$ and $\PP$ are defined in Lemma~\ref{t:system}.

If $\mu=0$ and $\varphi(0)$ is in $\Su$ or $\Sd$, then 
\begin{equation}
   \label{e:solution}
\vligne{B_{\u\u}^{(a,b)} & A_{\u\d}^{(a,b)}  \\ B_{\d\u}^{(a,b)} &
  A_{\d\d}^{(a,b)}}
 = \UU (I+ \PP)^\# + \vw \cdot \veta\tr,
\end{equation}
where $\veta$ is the left eigenvector of $I+\PP$ for the eigenvalue 0,
\[
\vw = (\vect \eta\tr  \vbeta)^{-1}
(\vh - \UU (I+\PP)^\# \vbeta),
\]
and $\vbeta$ is defined in (\ref{e:c}).
\end{thm}

\begin{proof}
The first two statements have been
justified before, we include them for completeness.
If $\mu = 0$, the matrix $\PP$ is stochastic and
irreducible.  As
\[
I -\PP =
\vligne{I & -\Psi e^{U(a+b)} \\ -\whpsi e^{\whU(a+b)} & I}
 =
\vligne{I & 0 \\ 0 & -I}  (I+\PP) \vligne{I & 0 \\ 0 & -I},
\]
the matrices $I-\PP$ and $I+\PP$ are similar and $I+\PP$ has a unique
eigenvalue equal to 0.  The corresponding left eigenvector $\veta$ is such
that 
\begin{align}
  \label{e:etam}
\vect\eta\tr_\d & = -\vect\eta\tr_\u \Psi e^{U(a+b)},
\\
  \nonumber
\vect\eta\tr_\u   & = \vect\eta\tr_\u \Psi e^{U(a+b)} \whpsi e^{\whU(a+b)},
\end{align}
and we may choose $\vect\eta\tr_\u > \vzero$.
Thus, the system (\ref{e:system}) has the solution
\begin{equation}
   \label{e:sysba}
\vligne{B_{\u\u}^{(a,b)} & A_{\u\d}^{(a,b)}  \\ B_{\d\u}^{(a,b)} &
  A_{\d\d}^{(a,b)}} = \UU (I+\PP)^\# + \vw \cdot\veta\tr
\end{equation}
for some vector $\vw$.  
We post-multiply (\ref{e:sysba}) by $\vbeta$,  the left-hand
side is equal to $\vect h$ by Lemma \ref{t:null} and we obtain
\[
\vw = (\vect\eta\tr  \vbeta)^{-1} (\vect h - \UU (I+\PP)^\#  \vbeta),
\]
provided that $\vect\eta\tr  \vbeta \not = 0$.   This is equivalent to
showing that $\vbeta$ is, indeed, linearly
independent of the columns of $I+\PP$, which in turn implies that
(\ref{e:system}, \ref{e:supplement}) is a non-singular system when
$\mu = 0$.

Now,
\begin{align}
   \nonumber
\vect\eta\tr \vbeta& = b \vect\eta\tr_\u \vone_\u + \vect\eta\tr_\u \vect h_\u -
a \vect\eta\tr_\d \vone_\d + \vect\eta\tr_\d \vect h_\d \\
   \label{e:tha}
 & =  b \vect\eta\tr_\u \vone_\u +
a \vect\eta\tr_\u \vone_\u  +  \vect\eta\tr_\u (\vect h_\u - \Psi e^{U(a+b)}
\vect h_\d) 
\end{align}
by (\ref{e:etam}).   The vector $\vect h_\d$ is indexed by phases in
$\Sd$ and we write, for short,
\begin{align*}
\vect h_\d & = \lim_{t \rightarrow \infty} \E_0[X(t)| \varphi(0) \in
\Sd] \\
& = \lim_{t \rightarrow \infty} (\E_0[X(t) \indic\{\tau_0^\u < t \}| \varphi(0) \in \Sd]
+ \E_0[X(t) \indic\{\tau_0^\u \geq t \}| \varphi(0) \in \Sd]  )\\
& = \lim_{t \rightarrow \infty} (\P[\tau_0^\u < t, \varphi(\tau_0^\u)|\varphi(0) \in
\Sd] \, \E_0[X(t) |\tau_0^\u < t , X(\tau_0^\u), \varphi(\tau_0^\u)] \\
& \quad + \P[\tau_0^\u \geq t|\varphi(0) \in \Sd] \, \E_0[X(t) |\tau_0^\u
\geq t , \varphi(0) \in \Sd]) \\
& = \whpsi \vect h_\u+ \lim_{t \rightarrow \infty} \P[\tau_0^\u \geq
  t|\varphi(0) \in \Sd] \, \E_0[X(t) |\tau_0^\u
\geq t , \varphi(0) \in \Sd]) 
\intertext{as $X(\tau_0^\u)=0$ and $\varphi(\tau_0^\u) \in \Su$,}
& \leq \whpsi \vect h_\u
\end{align*}
since $X(t) < 0$ for $t < \tau_0^\u$.   The inequality is strict  for
at least one component of $\vh_\d$.
Otherwise, by Remark~\ref{r:tbd}, 
 we would get $\vect h_\d= \whpsi \vect h_\u = \whpsi
\Psi \vect h_\d$, from which we would successively conclude that $\vect
h_\d = c \vone$ for some scalar $c$, that $\vect h_\u = c \vone$, and that
$\vect h=c \vone$, which would be in contradiction with $\valpha\tr \vect
h = 0$, by Remark \ref{r:tbd}.   Consequently, (\ref{e:tha}) becomes
\begin{equation}
   \label{e:thb}
\vect\eta\tr \vbeta  >
(a +b) \vect\eta\tr_\u \vone_\u +
 \vect\eta\tr_\u (I - \Psi e^{U(a+b)}
\whpsi) \vect h_\u.
\end{equation}
By an argument similar to the one held above, we  find that 
\begin{align*}
\vect h_\u & = \lim_{t \rightarrow \infty} \E_0[X(t)-(a+b)| \varphi(0) \in
\Su] + (a+b)\vone\\
& = \lim_{t \rightarrow \infty} (\P_0[\tau_{a+b}^\u < t, \varphi(\tau_{a+b}^\u)|\varphi(0) \in
\Su] \E_0[X(t) -(a+b)|\tau_{a+b}^\u < t , \varphi(\tau_{a+b}^\u)] \\
 & \quad +(a+b) \vone\\
& \quad + \P_0[\tau_{a+b}^\u \geq t|\varphi(0) \in \Su] \E_0[(X(t)-(a+b)|\tau_{a+b}^\u
\geq t , \varphi(0) \in \Sd]) \\
& \leq e^{\whU(a+b)} \vect h_\u + (a+b) \vone
\end{align*}
as $X(\tau_{a+b}^\u) -(a+b)=0$, $\varphi(\tau_{a+b}^\u)
  \in \Su$, and $X(t) -(a+b) < 0$ for $t < \tau_{a+b}^\u$.
With this, (\ref{e:thb}) becomes
\begin{align*}
\vect\eta\tr \vbeta  & >
(a+b) \vect\eta\tr_\u \vone_\u +
 \vect\eta\tr_\u (I - \Psi e^{U(a+b)}
\whpsi e^{\whU (a+b)}) \vect h_\u - (a+b) \vect\eta\tr_\u\vone_\u \\
 & = 0.
\end{align*}
This completes the proof.
\end{proof}

In the same manner as in Section \ref{s:times}, we may determine the LS transform of
the random variable $\min(\tau_a^\d, \tau_b^\u)$, details are in Bean
{\it et al.}~\cite{bot09}.

\section{Numerical procedures}
\label{s:algorithms}

It should be clear by this point that the numerical evaluation of many
quantities of interest is dependent on being able to compute the
matrices $\Psi$ and $\whpsi$.  If we replace $U$ in (\ref{e:psia}) by
the right-hand side of (\ref{e:u}) and write
\[
\Psi = \int_0^\infty   e^{C_\u^{-1} T_{\u\u} y} \, C_\u^{-1} T_{\u\d} \,
e^{(|C_\d|^{-1} T_{\d\d}  + |C_\d|^{-1} T_{\d\u} \Psi )y} \, \ud y,
\]
then an obvious approach to compute $\Psi$ is to proceed by successive
substitution: define iteratively
\begin{equation}
   \label{e:functional}
\Psi_{n+1} = \int_0^\infty   e^{C_\u^{-1} T_{\u\u} y} \, C_\u^{-1} T_{\u\d} \,
e^{(|C_\d|^{-1} T_{\d\d}  + |C_\d|^{-1} T_{\d\u} \Psi_n )y} \, \ud y.
\end{equation}
for $n \geq 0$, starting from $\Psi_0 = 0$, The resulting sequence is
monotonically convergent to $\Psi$ as we show in
Theorem~\ref{t:functional}.  The proof is new, and we give it in
detail.  It is based on the evolution of a stack $\sigma$
associated to the fluid queue.  

At the epochs when the phase process enters $\Su$ after a sojourn in
$\Sd$, we put the level at the top of the stack; the value recorded on
top of the stack is removed when the fluid decreases to that level.
 Formally, we
define the sequence $\{s_k: 0 \leq k \leq L\}$ of epochs where the
stack increases during the interval $[0, \Theta]$:
\begin{align*}
s_0  &  = 0, 
\\
s_k & = 
\inf\{t > f_k : \varphi(t) \in \Su\} 
\qquad \mbox{with $f_k = \inf\{t > s_{k-1}: \varphi(t) \in \Sd\}$,}
\end{align*}
and $L = \sup\{k: s_k < \Theta\}$.  On the sample path of Figure
\ref{f:stack}, $L=6$ and we have marked $s_0$ to $s_6$.

\begin{figure}[t]
\centering{
\includegraphics{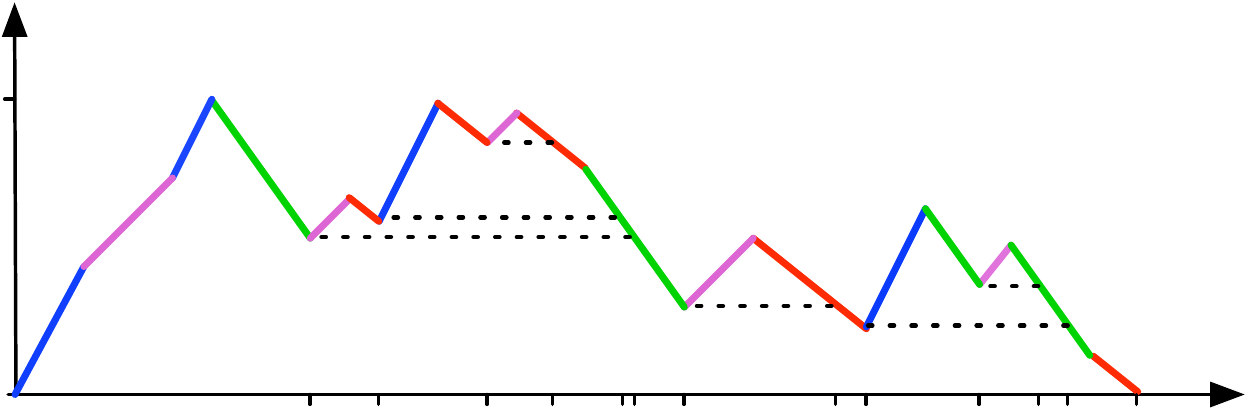}
\put(-10,-10){\makebox(0,0)[c]{$t$}}
\put(-30,-7){\makebox(0,0)[c]{$s'_0$}}
\put(-355,-8){\makebox(0,0)[c]{$s_0$}}
\put(-268,-8){\makebox(0,0)[c]{$s_1$}}
\put(-250,-8){\makebox(0,0)[c]{$s_2$}}
\put(-220,-8){\makebox(0,0)[c]{$s_3$}}
\put(-165,-8){\makebox(0,0)[c]{$s_4$}}
\put(-110,-8){\makebox(0,0)[c]{$s_5$}}
\put(-78,-8){\makebox(0,0)[c]{$s_6$}}
\put(-375,110){\makebox(0,0)[c]{$X(t)$}}
\put(-367,90){\makebox(0,0)[c]{$y$}}
\put(-367,5){\makebox(0,0)[c]{$0$}}
}
\caption{The fluid level $X(s_i)$ is pushed on the stack at time $s_i$,
  $i = 0, 1 \ldots$.  The epochs $s'_1$ to $s'_6$ of removal of the top are marked but
  not labeled, they are ordered as $s'_3 < s'_2<s'_1 < s'_4 < s'_6 <
  s'_5 < s'_0$.   The maximum size of the stack is 4, reached at time $s_3$.}
\hrulefill
\label{f:stack}
\end{figure}
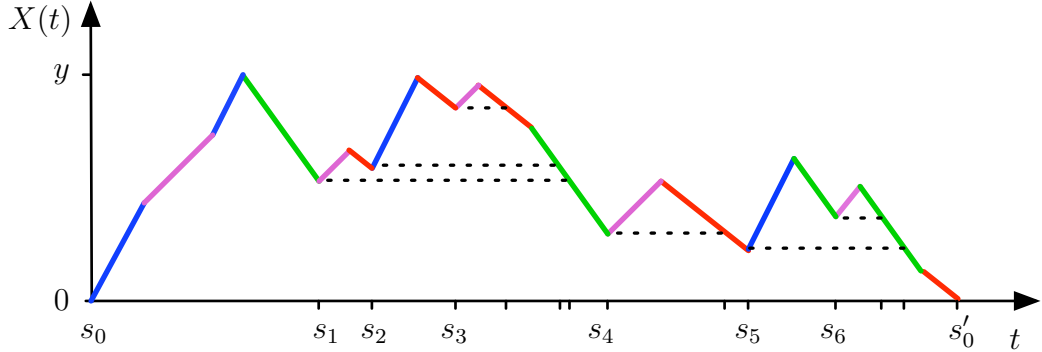

Next, we define $\{s'_k : 0 \leq k \leq L\}$:
\[
s'_k = \inf\{t > s_k : X(t) = X(s_k)\}.
\]
On Figure~\ref{f:stack}, the epochs $s'_1$ to $s'_6$ are marked, but
  not labeled so as not to clutter the graph.

  At time $s_k$, the size $|\sigma|$ increases by one and we record
  $X(s_k)$ on top of $\sigma$, at time $s'_k$ we remove the top of
  $\sigma$ and $|\sigma|$ decreases by one.  Note that
  $s'_0 = \Theta$, and that the stack becomes empty for the first
  time.

\begin{thm}
   \label{t:functional}
The sequences $\Psi_n$, $n \geq 0$ defined by
\begin{equation}
   \label{e:stack}
(\Psi_n)_{ij} = \P[\Theta < \infty, \varphi(\Theta)=j, \max_{0 \leq t
  \leq \Theta} |\sigma(t)|\leq n | \varphi(0)=i],
\end{equation}
for $i$ in $\Su$, $j$ in $\Sd$, satisfies (\ref{e:functional}).
Furthermore, $\Psi_n$ is the unique solution of the linear equation
\begin{equation}
   \label{e:psin}
 C_\u^{-1} T_{\u\d}  +     C_\u^{-1} T_{\u\u} \Psi_{n}  + \Psi_{n} U_{n-1} = 0,
\end{equation}
where
\begin{equation}
   \label{e:un}
 U_{n-1}  = |C_\d|^{-1} T_{\d\d}  + |C_\d|^{-1} T_{\d\u} \Psi_{n-1}
\end{equation}
for $n \geq 1$.

The sequence $\{\Psi_n\}$ converges monotonically to $\Psi$ and
$\{U_n\}$ converges to $U$, as $n \rightarrow \infty$.
\end{thm}

\begin{proof}
It is obvious that the sequence defined by (\ref{e:stack}) is monotone
and converges to $\Psi$ as there are fewer constraints on the
trajectories for increasing $n$ until there is none in the limit.

Next, we show that the return probabilities defined in (\ref{e:stack})
are solutions of (\ref{e:functional}).  For $n=1$, we have 
\[
\Psi_1 = \int_0^\infty   e^{C_\u^{-1} T_{\u\u} y} \, C_\u^{-1}
T_{\u\d} \,
e^{|C_\d|^{-1} T_{\d\d}y} \, \ud y.
\]
This means  that the fluid queue spends some time in $\Su$ and
grows up to some level $y$, then switches to $\Sd$ and never returns
to $\Su$ until hitting level 0.   We push the value 0 on $\sigma$ at
time 0, remove it at time $\Theta$ and so $|\sigma| = 1$ over the
whole interval.  This shows that $\Psi_1$ is a solution of
(\ref{e:functional}) with $\Psi_0=0$.

For the general case, we illustrate on Figure \ref{f:psin} the
physical meaning of the right-hand side of (\ref{e:functional}): the
fluid process grows up to some level $y$, then goes down to 0 with
occasional episodes of growth; during those episodes, the trajectories
followed by the process are constrained by the definition of $\Psi_n$.

\begin{figure}[t]
\centering{
\includegraphics{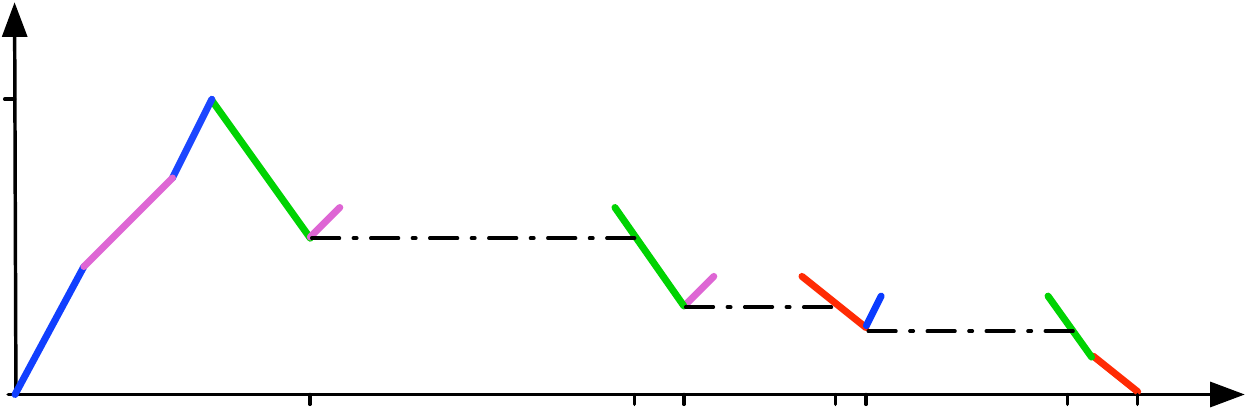} 
\put(-10,-10){\makebox(0,0)[c]{$t$}}
\put(-30,-7){\makebox(0,0)[c]{$s'_0$}}
\put(-355,-8){\makebox(0,0)[c]{$s_0$}}
\put(-268,-8){\makebox(0,0)[c]{$a_1$}}
\put(-175,-7){\makebox(0,0)[c]{$b_1$}}
\put(-222,65){\makebox(0,0)[c]{$\Psi_n$}}
\put(-160,-8){\makebox(0,0)[c]{$a_2$}}
\put(-120,-7){\makebox(0,0)[c]{$b_2$}}
\put(-140,45){\makebox(0,0)[c]{$\Psi_n$}}
\put(-105,-8){\makebox(0,0)[c]{$a_3$}}
\put(-50,-7){\makebox(0,0)[c]{$b_3$}}
\put(-80,37){\makebox(0,0)[c]{$\Psi_n$}}
\put(-375,110){\makebox(0,0)[c]{$X(t)$}}
\put(-367,90){\makebox(0,0)[c]{$y$}}
\put(-367,5){\makebox(0,0)[c]{$0$}}
}
\caption{Approximation $\Psi_{n+1}$ by functional iteration: the size
  of the stack increases at most by $n$ over
  the intervals $[a_1, b_1]$, $[a_2, b_2]$ and $[a_3, b_3]$ and is at
  most $n+1$ over the whole interval $[s_0, s'_0]$.}
\hrulefill
\label{f:psin}
\end{figure}
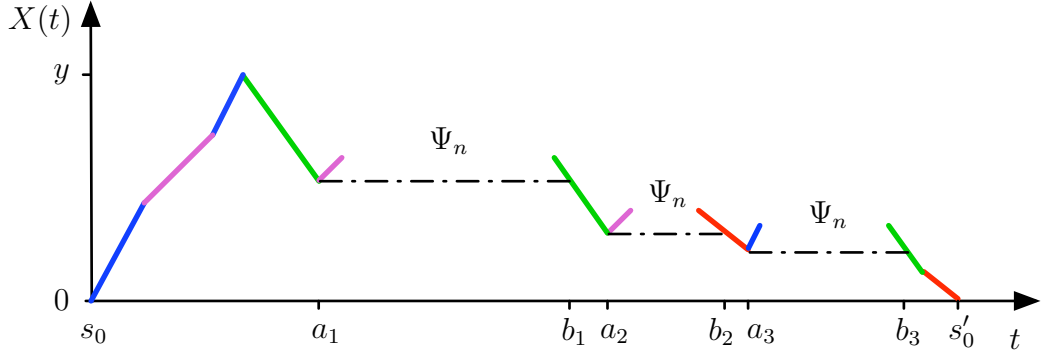

Such episodes, if any, occur during intervals $(a_i,b_i)$, $1 \leq i
\leq L'$, with
\begin{align*}
a_1 & = s_1,    \\
b_i &= \inf\{ t > a_i : X(t)= X(a_i)\}  \\
a_{i+1}  & = \inf\{t > b_i: \varphi(t) \in \Su\},
\end{align*}
$L' = \sup\{i : a_i < \Theta\}$.  During the intervals $(s_0,a_1) \cup 
(\cup_{1 \leq i \leq L'-1} (b_i,a_{i+1}) \cup (b_{L'}, s'_0)$,
$\sigma$ contains only the value 0 which was pushed at time 0, and
$|\sigma|=1$.  During the intervals $(a_i, b_i)$, the stack may
increase at most by $n$ units.  Thus, $|\sigma| \leq n+1$ over the
whole interval $[0, \Theta]$ and the right-hand side of
(\ref{e:functional}) is equal to $\Psi_{n+1}$. 

With $U_n$ defined in (\ref{e:un}), (\ref{e:functional}) becomes 
\begin{equation}
\Psi_{n+1} = \int_0^\infty   e^{C_\u^{-1} T_{\u\u} y} \, C_\u^{-1} T_{\u\d} \,
e^{U_ny} \, \ud y.
\end{equation}
and (\ref{e:psin}) follows in the same manner as we proved
(\ref{e:psib}).  The coefficients $C_\u^{-1} T_{\u\u}$ and $U_n$ are
both defective generators, all of their eigenvalues are in $\CC_{<0}$,
and so the system (\ref{e:psin}) has a unique solution $\Psi_n$.  Therefore,
the sequence $\{U_n, \Psi_n\}$ is well-defined and the theorem
follows.
\end{proof}

The algorithm defined by (\ref{e:psin}, \ref{e:un}) is easily
implemented and is the most efficient among several {\em linearly} convergent
algorithms, as shown in Bean {\it et al.}~\cite{brt05b}.   Several
other procedures have been proposed in Guo~\cite{guo01b} and Bini {\it et
  al.}~\cite{bilm04, bim12}.  
A special mention should be made of the Newton method
as it is  easily implemented and much faster than functional
iteration: the sequence $\{\Psi^{(N)}_n\}$ defined by
\begin{align*}
  C_\u^{-1} T_{\u\d}  &+     (C_\u^{-1} T_{\u\u} + \Psi^{(N)}_n
  |C_\d|^{-1} T_{\d\u}) \Psi^{(N)}_{n+1}  
\\ & + \Psi^{(N)}_{n+1}
  (|C_\d|^{-1} T_{\d\d}  + |C_\d|^{-1} T_{\d\u}  \Psi^{(N)}_n )  = 
\Psi^{(N)}_n  |C_\d|^{-1} T_{\d\u}  \Psi^{(N)}_n 
\end{align*}
for $n \geq 0$, with $\Psi^{(N)}_0  = 0$,
is well-defined and converges  {\em quadratically} if
$\mu \not= 0$.

The most efficient algorithms today form the
family of doubling algorithms, which solve simultaneously for $\Psi$
and for $\whpsi$; they include the structure-preserving doubling
algorithm (SDA, Guo {\it et al.}~\cite{glx06}), SDA shrink-and-shift
(Bini {\it et al.}~\cite{bmp10}), componentwise-accurate doubling
algorithms (Nguyen and Poloni~\cite{np15}), and
alternating-directional doubling algorithm (Wang {\it et
  al.}~\cite{wwl12}). These algorithms are quadratically
convergent if $\mu \neq 0$ and, furthermore, each iteration is 
faster due to fewer computations.

\section{Extensions}
\label{s:extensions}

After the publication of Ramaswami's seminal paper \cite{ram99}, 
the basic fluid queue model defined through Equations (\ref{e:x}) and (\ref{e:regulated}) has been
extended in many ways.  We cite some of these, without getting into
details,  the  list is far from exhaustive.

\paragraph{Finite buffers}
In some applications, $Y(t)$ represents the content of some finite
buffer, so that the level may only take values in some interval $0
\leq Y(t) \leq B < \infty$ (see da Silva Soares and Latouche~\cite{dssl05,dssl03}
and Bean {\it et al.}~\cite{bot09}).  
In that case, it is natural to
choose as regeneration points the epochs of return to the upper
boundary in addition to the returns to level 0, and the transition
matrix $H$ between regeneration points makes use of
first passage probabilities from boundary to boundary.  These may be obtained from Lemma~\ref{t:system} with
$a=0$ and $b=B$,  if the process starts from the lower boundary, or
$a=B$ and $b=0$, if the process starts from the upper boundary.

The final expression for the stationary density is a mixture of two matrix
exponentials, $e^{Kx}$ and $e^{\whK(B-x)}$.  Details may be found in
\cite{dssl07} and in \cite{ln15b}.

\paragraph{Level feedback}

In our presentation so far, the level is driven by the phase, subject
to the boundary constrains that $Y(t) \geq 0$, or $0 \leq Y(t) \leq
B$, while the phase evolves independently of the level.  In many
cases, the level has a direct influence on the evolution of the phase;
for instance, one might reduce the flow into the buffer as the level gets
nears the upper boundary, so as to avoid spillage.  

In Bean {\it et al.}~\cite{bo08} and da Silva Soares~\cite{dssl07}
(and other references cited there), one defines a number of threshold
values $0 \leq b_1 < b_2 < \cdots < b_N \leq \infty$ such that the
parameters $C$ and $Q$ of the phase process change upon the crossing of a
threshold.  Here, regenerations occur when the fluid reaches any of
the threshold, and the analysis of such systems may be based on a
systematic extension of the results for the system with a finite buffer.

\paragraph{Fluid with jumps}
In Remiche~\cite{remi04} $Y(t)$
represents the supply of tokens in a leaky bucket system.  It increases
linearly in time and drops by a positive amount each time a file is
transmitted.  
In Bean {\it et al.}~\cite{bos10}, $Y(t)$ represents the amount of
wear of a power generator and it may jump from $B$ (indicating that
the generator is unusable) to 0 (indicating that it has been replaced
by a new equipment).  
As discussed in Badescu {\it et
  al.}~\cite{bbdls04} and Stanford {\it et al.}~\cite{sabbdl05},
risk processes may be analysed as fluid
queues with jumps.

If the jumps have a
phase-type distribution, then the analysis of the process requires
very little adaptation from the material presented in the present paper.

\paragraph{Markov modulated Brownian motion}  

The definition of these processes is very similar to that of fluid
flows.  The difference is that the fluid evolves like a Brownian
motion with parameters (drift and variance) which depend on
$\varphi(t)$.  Recent references are 
d'Auria {\it et al.} \cite{aikm10,aikm12}, Ivanovs~\cite{ivano10},
Gribaudo {\it et al.}~\cite{gmst08} where the authors focused on
obtaining time-dependent distributions and first hitting times using
different  approaches: stochastic ODE resolution, spectral decomposition
and martingale theory.   Breuer~\cite{breue12} determined the occupation time of the process in an interval before a one- or two-sided exit.
Latouche and Nguyen~\cite{ln15b,ln15} are two recent papers that
follow a regenerative approach similar to the one developed here. 

\paragraph{Two-dimensional fluid}
 
A few authors have considered systems where the component $X$ is
two-dimensional: Bean and O'Reilly~\cite{bo13,bo14}, Foss and
Miyazawa~\cite{fm14} and Latouche {\it et al.}~\cite{lnp11b}
among others.  The area of two- or higher-dimensional fluids is still
wide open, with many exciting unanswered questions.  

\subsubsection*{Acknowledgements}
The authors would like to acknowledge the support of ACEMS (Australian Research Council Centre of Excellence for Mathematical and Statistical Frontiers).

\bibliographystyle{abbrv}
\bibliography{gnral}

\end{document}